\documentclass[11pt, a4paper]{article}
\usepackage{mathrsfs}
\usepackage{CJK}
\usepackage{bbm}
\usepackage{stmaryrd}
\usepackage{amsmath}
\usepackage{cases}
\usepackage{amsfonts}
\usepackage{latexsym,amssymb,amsmath,mathrsfs,amsbsy, amsthm}
\usepackage[usenames]{color}
\usepackage{xspace,colortbl}
\usepackage[dvipdfm,
pdfstartview=FitH, CJKbookmarks=true, bookmarksnumbered=true,
bookmarksopen=true,
citecolor=blue %¸Ä±äĿ¼¡¢ÒýÓõÄÑÕÉ«
,colorlinks=true
,pdfborder=1 %×¢Ê͵ô´ËÏîÔò½»²æÒýÓÃΪ²ÊÉ«±ß¿ò
,pdfstartpage=2 %´ò¿ªÊ±ÎªµÚ2Ò³
,pdfstartview=Fit]{hyperref} %´ò¿ªpdfʱΪ fit page
\allowdisplaybreaks

\newtheorem{thm}{Theorem}[section]
\newtheorem{lem}[thm]{Lemma}

\newtheorem{rem}{Remark}
\newtheorem{coro}[thm]{Corollary}

\newtheorem{defnm}{Definition}[section]
\newtheorem{prop}[thm]{Proposition}
\newcommand{\nn}{\nonumber}

\newcommand{\la}{\langle}
\newcommand{\ra}{\rangle}

\newcommand{\al}{\alpha}
\newcommand{\nb}{\nabla}
\newcommand{\ro}{\rho}
\newcommand{\ld}{\lambda}
\newcommand{\epa}{\varepsilon}

 \newcommand{\Norm}[1]{\left\Vert#1\right\Vert}
\numberwithin{equation}{section}
\begin{document}
\title{\bf New Geometric Flows on Riemannian Manifolds and Applications to
Schr\"odinger-Airy Flows}

\author{\bf Xiaowei Sun\thanks{Supported by NSFC, Grant No. 11226082} and Youde
Wang\thanks{Supported by NSFC, Grant No. 10990013}}

\date{}
\maketitle

\begin{abstract}
In this paper, we define a class of new geometric flows on a
complete Riemannian manifold. The new flow is related to the
generalized (third order) Landau-Lifishitz equation. On the other
hand it could be thought of a special case of the Schr\"odinger-Airy
flow when the target manifold is a K\"ahler manifold with constant
holomorphic sectional curvature. We show the local existence of the
new flow on a complete Riemannian manifold with some assumptions on
Ricci tensor. Moreover, if the target manifolds are Einstein or some
certain type of locally symmetric spaces, we obtain the global
results.
\end{abstract}

%\newpage
\section{Introduction}
Let $(N, h)$ be a Riemannian manifold equipped with a metric $h$.
Then there is a Levi-Civita connection associated to $h$. Denote the
Levi-Civita connection on $(N, h)$ by $\nabla$  and the
corresponding curvature operator on $(N, h)$ by $R$, which is a
$(1,3)$-tensor.

The Ricci curvature $Ric$ is the trace of the curvature operator
$R$. It is well-known that $Ric$ is a symmetric bilinear form. If
for any tangent vector fields $X, Y$ on $(N, h)$ there holds
$$Ric(X, Y)=kh(X, Y),$$ then $(N, h)$ is said to be an Einstein
manifold with Einstein constant $k$.

For any smooth map $u(x,t)$ from $S^1\times\mathbb{R}$ into $(N,h)$,
Let $\nabla_x$ denote the covariant derivative
$\nabla_{\frac{\partial}{\partial x}}$ on the pull-back bundle
$u^{-1}TN$ induced from the Levi-Civita connection $\nabla$ on $N$.
For the sake of convenience, we always denote $\nabla_xu$ and
$\nabla_tu$ by $u_x$ and $u_t$ respectively.

For the maps from a unit circle $S^1$ or a real line $\mathbb{R}$
into $N$, we define a class of new geometric flows as follows:
\begin{eqnarray}
\frac{\partial u}{\partial t} = \nabla_x^2 u_x
+{\ro}Ric(u_x,u_x)u_x,\nn
\end{eqnarray}
where $\ro$ is a positive constant. If $(N, h)$ is an Einstein
manifold, the new geometric flow is an energy conserved system.

Obviously, if the target manifold is Euclidian space $\mathbb{R}^n$,
the new flow then reduces to the vector Airy equation
$$u_t=u_{xxx}.$$
We should mention that solutions to the scalar Airy equation may be
expressed in terms of so-called Airy functions, which are named
after the astronomer George Biddell Airy(1801-1892) and are
solutions to the differential equation $$y_{xx}(x)-xy(x)=0.$$

One of the related problems with our new flow is the so-called
generalized Landau-Lifishitz equation written by
\begin{eqnarray*}
u_t =  u_{xxx} + \frac{3}{2}(|u_x|^2u)_x, ~~~~~ |u|^2 = 1.
\end{eqnarray*}
Here $u: S^1(~\mbox{or}~~\mathbb{R})\times\mathbb{R}\rightarrow
S^n\subset\mathbb{R}^{n+1}$ and $|\cdot|$ denotes the standard
metric in $\mathbb{R}^{n+1}$ (see \cite{DJKM, MS, TS}). If $n=2$, it
is just the third order Landau-Lifshitz equation. It is easy to see
that this equation can be rewritten by a geometric version as
follows:
\begin{eqnarray*}
u_t = \nabla_x^2 u_x + \frac{1}{2}|u_x|^2u_x,
\end{eqnarray*}
where $\nabla_x=\nabla_{\frac{\partial}{\partial x}}$ denotes the
Levi-Civita connection on the standard unit sphere $S^n$. For detail
we refer to \cite{SY}. Furthermore, we may intrinsically write the
equation by
\begin{eqnarray*}
u_t = \nabla_x^2 u_x + \frac{1}{2(n-1)}Ric(u_x, u_x)u_x,
\end{eqnarray*}
where $Ric(\cdot,\cdot)$ is the Ricci tensor (operator) on $S^n$.

In fact, from the integrable system point of view the following
generalized Landau-Lifshitz equation was considered in \cite{GS, MS,
TW}
\begin{eqnarray}\label{eq:000}
u_t = \nabla_x^2 u_x + \frac{1}{2}|u_x|^2u_x + \frac{3}{2}\langle u,
A(u)\rangle u_x,
\end{eqnarray}
where $u:\mathbb{R}\times\mathbb{R}\rightarrow
S^n\subset\mathbb{R}^{n+1}$ and $A$ is a constant symmetric
$(n+1)\times(n+1)$ matrix. It was shown that this equation is
integrable by the inverse scattering method for any $n$ and $A$. One
has known that this equation also defines an infinitesimal symmetry
for the well-known C. Neumann system \cite{MS, Ve}
$$u_{xx} + |u_x|^2 u= A(u) -\langle u, A(u)\rangle u, ~~~~~~ |u|^2=1$$
describing the dynamics of a particle on the unit sphere under the
influence of field with the quadratic potential
$$P =\frac{1}{2}\langle u, A(u)\rangle.$$

Recently, Song and Yu  in \cite{SY} employed the geometric energy
method established in \cite{DW, YD} to show the global
well-posedness of the corresponding Cauchy problem from
$S^1\times\mathbb{R}\rightarrow S^n$. We also mention that the
stationary solution of the geometric flow defines an interesting
kind of curves on $S^n$ (see \cite{So}).

On the other hand, if $(N, J, h)$ is a K\"ahler manifold, where $J$
is a compatible complex structure with $h$, in \cite{SW*} we
introduced a class of geometric flows for the maps from a unit
circle $S^1$ or a real line $\mathbb{R}$ into $(N, J, h)$, geometric
Schr\"odinger-Airy flow, as follows:
\begin{eqnarray}\label{eq:0.1*}
\frac{\partial u}{\partial t} = \al\left( \nabla_x^2 u_x
+{1\over2}R(u_x,J_uu_x)J_uu_x\right) + \beta J_u\nb_xu_x + \gamma
|u_x|^2u_x,
\end{eqnarray}
where $\al$, $\beta$ and $\gamma$ are real constants, $R$ is the
Riemannian curvature tensor on $N$ and $J_u\equiv J(u)$. In
\cite{SW*} we have shown that the Schr\"odinger-Airy flows relate
closely to several important and well-known physical or mechanical
systems \cite{FM,GK,GS,Ha,Hi,HK,MR,NT,Od,RL}. This flow is hybrid of
geometric KdV flow \cite{DQ, SW} and geometric Schr\"odinger flow
\cite{BIKT, Uh, CT, D, DW,RRS}, and relates closely to derivative
nonlinear Schr\"odinger equation\cite{SW*,WH}.

If $\beta=0$, the Schr\"odinger-Airy flow on a K\"ahler manifold
with constant holomorphic sectional curvature $K$ becomes
\begin{eqnarray}\label{eq:0.2*}
\frac{\partial u}{\partial t} = \al\left( \nabla_x^2 u_x
+{1\over2}K|u_x|^2u_x\right) + \gamma |u_x|^2u_x.
\end{eqnarray}
By scaling with respect to the time variable, we can change the
above flow into
\begin{eqnarray*}
\frac{\partial u}{\partial t} =  \nabla_x^2 u_x + \ro|u_x|^2u_x.
\end{eqnarray*}
This is just our new geometric flow on a K\"ahler manifold with
constant holomorphic sectional curvature. However, by our knowledge
one has not obtained any global existence results for
(\ref{eq:0.2*}). Thus, by choosing $N$ specially, the new flow could
be regarded as special cases of the Schr\"odinger-Airy flow. But in
general cases, these two flows differ a lot since the new geometric
flow is defined on all Riemannian manifolds while the
Schr\"odinger-Airy flow is only defined on K\"ahler manifolds.

In this paper, we mainly discuss the local existence for the Cauchy
problem of the new geometric flow on a complete Riemannian manifold
$(N,h)$ defined by
\begin{equation}\label{eq:0.1}
\left\{
\begin{aligned}
&u_t=\nabla_x^2u_x+{\ro}Ric(u_x,u_x)u_x,
\quad x\in S^1;\\
&u(x,0)=u_0(x).
\end{aligned}\right.
\end{equation}
Furthermore, when $N$ is some kind of special locally symmetric
spaces, we could obtain some results on global existence of
(\ref{eq:0.1}). The method we use here is the geometric energy
method which is also adopted to discuss the KdV geometric flow in
\cite{SW}. But technically speaking, the processes differ greatly
especially in proving the existence of the two geometric flows. By
utilizing the Ricci curvature tensor, we introduce a new geometric
norm which would help us to obtain the estimates we need.

Before stating our main results, we need to introduce several
definitions on Sobolev spaces of sections with vector bundle value
on $M$. Let $(E, M, \pi)$ be a vector bundle with base manifold $M$.
If $(E, M, \pi)$ is equipped with a metric, then we may define
so-called vector bundle value Sobolev spaces as follows:

\begin{defnm}
$H^m(M, E)$ is the completeness of the set of smooth sections with
compact supports denoted by $\{s| \,\,s\in C_0^\infty(M, E)\}$ with
respect to the norm
$$\|s\|^2_{H^m(M, E)}=\sum_{i=0}^m \int_M |\nabla^i s|^2dM.$$
Here $\nabla$ is the connection on $E$ which is compatible with the
metric on $E$.
\end{defnm}

\begin{defnm} \label{def:2}Let $\mathbb{N}$ be the set of positive integers.
For $m \in\mathbb{N}\cup\{0\}$, the Sobolev space of maps from $S^1$
into a Riemannian manifold $(N, h)$ is defined by
$$H^{m+1}(S^1;N) = \{u\in C(S^1; N) |\,\, u_x\in H^m(S^1; TN)\},$$
where $u_x \in H^m(S^1; TN)$ means that $u_x$ satisfies
$$\|u_x\|^2_{H^m(S^1; TN)} =\sum^m_{j=0}\int_{S^1}h(u(x))(\nabla^j_xu_x(x),
\nabla_x^ju_x(x))dx < +\infty.$$
\end{defnm}

We usually use $W^{k,p}(M, N)$ to denote the space of Sobolev maps
from $M$ into $N$, and $W^{k,p}(M, \mathbb{R}^l)$ to denote the
space of Sobolev functions.\\

\noindent Our main results are as follows:

\begin{thm}\label{thm:1.1}
Let $(N, h)$ be a complete Riemannian manifold with parallel Ricci
tensor, i.e. $\nb Ric\equiv0$. If the Ricci curvature on $N$ has a
positive lower bound $\lambda>0$ (or a negative upper bound
$-\lambda<0$), then the local solutions $u\in L^\infty
([0,T],H^k(S^1, N))$ $($$k \geq 4$$)$ of the Cauchy problem
(\ref{eq:0.1}) with the initial map $u_0\in H^k(S^1, N)$ is unique.
Moreover, the local solution is continuous with respect to the time
variable, i.e., $u\in C([0,T],H^k(S^1, N))$.

\end{thm}

\begin{thm}\label{thm:1.2}
Let $(N, h)$ be a complete Riemannian manifold with parallel Ricci
tensor, i.e. $\nb Ric\equiv0$ and the Ricci curvature on $N$ has a
positive lower bound $\lambda>0$ (or a negative upper bound
$-\lambda<0$). Then, for any integer $k \geq 4$ the Cauchy problem
of (\ref{eq:0.1}) with the initial value map $u_0\in H^k(S^1, N)$
admits a unique local solution $u\in C([0,T],H^k(S^1, N))$, where
$T=T(N,||u_0||_{H^4})$. Moreover, besides the assumptions on $N$, if
the Riemmainan curvature on $N$ satisfies $|\nb^lR|\leq
B_l(l=0,1,2,3)$ where $B_l$ are positive constants, then if the
initial value map $u_0\in H^3(S^1,N)$, the Cauchy problem of
(\ref{eq:0.1}) admits a local solution $u\in
L^\infty([0,T],H^3(S^1,N))$, where $T=T(N,||u_0||_{H^3})$.

\end{thm}

\begin{thm}\label{thm:1.3}
Assume that $(N,h)$ is a complete locally symmetric space on which
the Ricci curvature has a positive lower bound $\lambda>0$ (or a
negative upper bound $-\lambda<0$). Then for any integer $k\geq4$
the Cauchy problem (\ref{eq:0.1}) with the initial map $u_0\in
H^k(S^1,N)$ admits a unique global solution $u\in
C([0,\infty),H^k(S^1,N))$.
\end{thm}

It is worthy to point out that in this work, we still employ the
parabolic approximation and the geometric energy method developed in
\cite{DW,YD} to show the local existence problems. The process is
similar but different with that about the KdV geometric flow in
\cite{SW}. In fact, to show the local existence, we first obtain a
local solution $u_\epa$ of the following approximated problem
\begin{equation}\label{eq:0.5}
\left\{
\begin{aligned}
&u_t=-\epa\nabla_x^3u_x+\nabla_x^2u_x+{\ro}Ric(u_x,u_x)u_x,\quad x\in S^1;\\
&u(x,0)=u_0(x)\in H^k(S^1,N).
\end{aligned}\right.
\end{equation}
Then we have to derive the uniformly bound of $||u_{\epa
x}||^2_{H^m}$ which is independent of $\epa$ such that we could
obtain a limit $u(x,t)$ of the sequence $\{u_\epa\}$ in suitable
spaces as $\epa$ goes to zero and  it remains to show that the limit
$u(x,t)$ is a solution of the Cauchy problem (\ref{eq:0.1}).
However, because the different structure between the new flow and
KdV geometric flow, we could not obtain the bounds of $||u_{\epa
x}||^2_{H^m}$ by calculating ${d\over dt}||u_{\epa x}||^2_{H^m}$
directly as we did in \cite{SW}. Precisely, if we differentiate
$||\nb_x^2u_{\epa x}||^2_{L^2}$ with respect to $t$ and substitute
(\ref{eq:0.5}), after integrating by parts, one would get terms as
follows
$$\int_{S^1} h\left(\nb^3_xu_{\epa x},u_{\epa x}\right) Ric(\nb_x^2u_{\epa x},u_{\epa x})dx,$$
which could not be bounded by $||u_{\epa x}||^{2s}_{H^2}$ for $s\in
\mathbb{Z}^+$. Hence we have to try different ways to get those
estimations. We find that, for a Riemannian manifold $N$ with
parallel Ricci curvature, i.e. $\nb Ric\equiv0$, if the Ricci
curvature has a positive lower bound $\lambda>0$ (or a negative
upper bound $-\lambda<0$), the Ricci tensor $Ric(\cdot,\cdot)$ will
have very nice properties which are similar with that of the metric
$h(\cdot,\cdot)$. Instead estimating $||u_{\epa x}||_{H^k}^2$, we
could estimate
$$I_m(u_{\epa x})=\sum_{i=0}^m \int_{S^1} Ric(\nb_x^i u_{\epa x},\nb_x^i u_{\epa x})dx$$
and derive the uniform bounds of $I_m$ on a time interval $(0,T)$
where both the bounds and $T$ are independent of $\epa$. Then we
would obtain the uniform bounds of $||u_x||^2_{H^m}$ since
$$||u_x||^2_{H^m}\leq {1\over \lambda}I_m(u_{\epa x})\ \Big(\text{or}\ ||u_x||^2_{H^m}
\leq -{1\over \lambda}I_m(u_{\epa x})\Big).
$$
These estimations will be derived in next section. In one word, the
condition on the Ricci curvature of $N$ helps us obtain the uniform
estimations about the approximated solution and its high order
co-variant derivatives. Then, standard arguments are adopted to
derive the local existence of (\ref{eq:0.1}).
\begin{rem}
It is well known that all irreducible symmetric spaces are Einstein.
We should note that all the above results hold true on Einstein
manifolds with a positive Einstein constant or a negative Einstein
constant. However, if $N$ is Ricci flat, i.e. $Ric\equiv0$, the new
geometric flow then changed to
$${\partial u\over \partial t}=\nb_x^2u_x,$$
the method to discuss the existence in this work is ineffective. In
this case, we could only obtain the local existence of the new flow
via the same arguments as that in \cite{SW}.
\end{rem}
  To prove the global existence we need to
exploit some conservation laws and semi-conservation law. We define
\begin{eqnarray}
E_1(u)&\equiv&\int Ric(u_x,u_x)dx;\nn\\
 E_2(u)&\equiv&\int Ric(\nabla_x
u_x,\nabla_x u_x)dx-{\ro\over2}\int Ric( u_x,u_x)^2
dx;\nn\label{eq:0.18}\\
%\text{and}\ \ \ \ \ \ \quad\quad\ \ \ \ \ \ \ \ \ \ \ \ \nn&&\\
E_3(u)&\equiv&6\int Ric(\nb_x^2u_x,\nb_x^2u_x)dx-20\ro\int Ric(\nb_xu_x,u_x)^2dx\nn\\
&&{}-10\ro\int Ric(\nb_xu_x,\nb_xu_x)Ric(u_x,u_x)dx\nn\\
 &&{}-4\int Ric(\nb_xu_x,R(\nb_xu_x,u_x)u_x)dx.
\end{eqnarray}

If $N$ is a locally symmetric space, for the smooth solution $u$ to
the Cauchy problem (\ref{eq:0.1}) we will establish the following in
Sec.4:
\begin{eqnarray}
{d\over dt}E_1(u) = 0,\quad \quad\quad {d\over dt}E_2(u)=0.\nn
\end{eqnarray}
Moreover, if $N$ is a locally symmetric space on which the Ricci
curvature has a positive lower bound $\lambda>0$ (or a negative
upper bound $-\lambda<0$), then we have
\begin{eqnarray}\label{eq:0.21}
{d\over dt}E_3(u)&\leqslant&C(N,\lambda,E_1(u_0),E_2(u_0))(1+E_3).
\end{eqnarray}
We could make use of the above conservation laws with respect to
$E_1(u)$ and $E_2(u)$ to derive a uniform a priori bound of
$||\nabla_xu_x||_{L^2}$ independent of $T$. By virtue of
(\ref{eq:0.21}), we will obtain the global existence results.

This paper is organized as follows: In Section 2 we employ the
geometric energy method to establish the local existence of the new
 geometric flow. Since the conservation and
semi-conservation laws mentioned  before are crucial for us to
establish the global existence of the Cauchy problem of the
geometric flow. We give a detailed calculation in Section 3. The
global existence of the geometric flow on sepcial locally symmetric
spaces is proved in Section 4.

\section{Local Existence and Uniqueness} In this section we
establish the local existence and the uniqueness of solutions for
the Cauchy problem of the new geometric flow (\ref{eq:0.1}) on a
complete Riemannian manifold on which the Ricci curvature is
parallel and has a positive lower bound $\lambda$, i.e. $Ric\geq
\lambda>0$ (or a negative upper bound $Ric\leq -\lambda<0$). We
adopt the language that $Ric\geq \lambda>0$ if all eigenvalues of
$Ric(X)$ are $\geq \lambda$. In $(0,2)$ language this means more
precisely that $Ric(X,X)\geq \lambda h(X,X)$ for all $X\in TN$.
Moreover, if $Ric\geq \lambda>0$ on a complete Riemannian manifold
$N$, then by Myers-Cheng's theorem we have $N$ is compact.

As in \cite{SW}, to show the local existence of (\ref{eq:0.1}), we
use the approximate method and discuss the following Cauchy problem:
\begin{equation}\label{eq:2.1}
\left\{
\begin{aligned}
&u_t=-\epa\nabla_x^3u_x+\nabla_x^2u_x+{\ro}Ric(u_x,u_x)u_x,\quad x\in S^1;\\
&u(x,0)=u_0(x),
\end{aligned}\right.
\end{equation}
where $\epa>0$ is a small positive constant.

We could imbed $N$ into a Euclidean space $\mathbb{R}^n$ for some
large positive integer $n$. Then $N$ could be regarded as a
submanifold of $\mathbb{R}^n$ and $u:S^1\times \mathbb{R}\rightarrow
N\subset \mathbb{R}^n$ could be represented as $u=(u^1,\cdots,u^n)$
with $u^i$ being globally defined functions on $S^1$  so that the
Sobolev-norms of $u$ make sense. We have
\begin{eqnarray*}
||u||^2_{W^{m,2}}=\sum_{i=0}^m||D^iu||^2_{L^2},
\end{eqnarray*}
where $D$ denotes the covariant derivative for functions on $S^1$.
 The equation (\ref{eq:2.1}) then becomes a fourth order
parabolic system in $\mathbb{R}^n$. In the appendix of \cite{SW}, we
have shown that the parabolic equation admits a local solution
$u_\epa\in C([0, T_\epa), W^{k,2}(S^1, N))$ if the initial value map
$u_0\in W^{k,2}(S^1, N)$ where $k\geq3$.

Thus, in order to show the local existence of (\ref{eq:0.1}), we
would find a uniform positive lower bound $T$ of $T_\epa$ and
uniform bounds for various norms of $u_\epa(t)$ in suitable spaces
for $t$ in the time interval $[0,T)$. Once we get these bounds it is
easy to check that $u_\epa$ subconverge to a strong solution of
(\ref{eq:0.1}) as $\epa\rightarrow 0$ via standard arguments.

Before proving the local existence, we shall introduce the
properties about the Ricci curvature tensor and the Riemannina
curvature $R$ on a Ricci parallel Riemannian manifold. We have
\begin{prop} Let $N$ be a Riemannian manifold with parallel Ricci
curvature tensor, i.e. $\nb Ric\equiv0$. Then for $X,Y,Z,W\in
\Gamma(TN)$, the Ricci curvature tensor and the Riemannian curvature
tensor satisfy the following properties:
\begin{eqnarray}\label{eq:3.0}
&& (1)\ Z \Big(Ric(X,Y)\Big)=Ric(\nb_ZX,Y)+Ric(X,\nb_ZY);\nn\\
&& (2)\ Ric( X,R(Z,W)Y)=-Ric( X,R(W,Z)Y)=Ric( Z,R(X,Y)W).
\end{eqnarray}
\end{prop}
These properties will be adopted frequently in the calculation
throughout this paper. The process to show them is almost the same
with proof of the symmetric and skew-symmetric properties of
$R(X,Y,W,Z)$, we omit the details. Note that if $X\in
\Gamma(u^{-1}TN)$ we have in local coordinates
$$(\nabla_x X)^\alpha= {\partial X^\alpha\over \partial
x}+\Gamma^\alpha_{\beta\gamma}(u){\partial u^\beta\over \partial
x}X^\gamma $$ and for $X=u_x$ we have
$$(\nabla_t u_x)^\alpha= {\partial^2 u^\alpha\over {\partial t\partial x}}
+\Gamma^\alpha_{\beta\gamma}(u){\partial u^\beta\over \partial
t}{\partial u^\gamma \over \partial x}.$$ It is easy to see that
$\nabla_t u_x=\nabla_x u_t$.\\

Now we start the proof of the local existence of (\ref{eq:0.1}).
Here we mainly discuss the case that $Ric\leq-\lambda<0$ on $N$
here. For the case $Ric\geq\lambda>0$, we could get the same results
via easier arguments.

To begin with, let $u=u_\epa$ be a solution of (\ref{eq:2.1}). We
have the following results:
\begin{lem}\label{lm:2.1}
(i) Assume that $N$ is a complete Riemannian manifold with parallel
Ricci tensor, (i.e. $\nb Ric\equiv0$). If $N$ has negative upper
bounds on the Ricci curvature ($Ric\leq -\lambda<0$) and uniform
bounds on the curvature tensor $R$ and its covariant derivatives of
any order $($i.e., $|\nabla^lR|\leq B_l$, $l=0, 1, 2, \cdots$$)$,
and $u_0\in H^k(S^1,N)$ with an integer $k\geq3$. Then there exists
a constant $T=T(||u_0||_{H^3})$, independent of $\epa\in(0,1)$, such
that if $u\in C([0,T_\epa),H^k(S^1,N))$ is a solution of
(\ref{eq:2.1}) with $\epa\in(0,1)$, then $T(||u_0||_{H^3})\leq
T_\epa$ and $||u(t)||_{H^{m+1}}\leq C(||u_0||_{H^{m+1}})$ for any
integer $2\leq m\leq k-1$.

(ii) Assume that $N$ is a complete Riemannian manifold with parallel
Ricci tensor and $N$ has negative upper bounds on the Ricci
curvature. Let $u_0\in H^k(S^1,N)$ with an integer $k\geq5$. Then
there exists a constant $T=T(||u_0||_{H^5})>0$, independent of
$\epa\in(0,1)$, such that if $u\in C([0,T_\epa),H^k(S^1,N))$ is a
solution of (\ref{eq:2.1}) with $\epa\in(0,1)$, then
$T(||u_0||_{H^5})\leq T_\epa$ and $||u(t)||_{H^{m+1}}\leq
C(||u_0||_{H^{m+1}})$ for any integer $2\leq m\leq k-1$.
\end{lem}

\begin{proof}
First fix a $k\geq3$ and let $m$ be any integer with $2\leq m\leq
k-1$. We may assume that $u_0$ is $C^\infty$ smooth. Otherwise, we
always choose a sequence of smooth functions $\{u^i_0\}$ such that
$u^i_0\rightarrow u_0$ with respect to the norms
$\|\,\cdot\,\|_{H^k}$ where $k\geq 3$.

As $N$ may not be compact we let, we let $\Omega\triangleq \{p\in
N:\text{dist}_N(p,u_0(S^1))<1\}$, which is an open subset of $N$
 with compact closure $\overline{{\Omega}}$. Let
$$T'=\sup\{t>0:u(S^1,t)\subset\Omega\}.$$
Now we prove that if $k=3$, for all $t\in[0,T_\epa]$
\begin{eqnarray}\label{eq:2.0}
{d\over dt}\left(-{1\over\lambda}\sum_{s=0}^2\int
Ric(\nabla_x^su_x,\nabla_x^su_x)dx\right)\leq{C(\Omega,\lambda)}\sum_{l=2}^4\left(-{1\over\lambda}\sum_{s=0}^2\int
Ric(\nb_x^su_x,\nb_x^su_x)dx\right)^l.
\end{eqnarray}
To see this, we first differentiate $\int Ric(u_x,u_x)dx$ with
respect to $t$ and we have:
\begin{eqnarray}
&&{d\over dt}\int Ric(u_x,u_x)dx\nn\\
&=&2 \int Ric(\nb_t u_x,u_x)dx=2\int Ric(\nb_x u_t,u_x)dx.\nn
\end{eqnarray}
Integrations by parts and substituting (\ref{eq:2.1}) yields
\begin{eqnarray}
&&{d\over dt}\int Ric(u_x,u_x)dx=-2\int Ric(u_t,\nb_x u_x)dx\nn\\
&=&2\epa\int Ric(\nb_x^3u_x,\nb_x u_x)dx-2\int Ric(\nb_x^2u_x,\nb_x
u_x)dx\nn\\
&&{}-2\ro\int Ric(Ric(u_x,u_x)u_x,\nb_x u_x)dx\nn\\
&=&-2\epa\int Ric(\nb_x^2u_x,\nb_x^2 u_x)dx-2\int
\nb_x\big(Ric(\nb_xu_x,\nb_x u_x)\big)dx\nn\\
&&{}-{\ro\over2}\int \nb_x\big(|Ric(u_x,u_x)|^2\big)dx\nn\\
&=&-2\epa\int Ric(\nb_x^2u_x,\nb_x^2
u_x)dx\nn%\leq-2\epa\lambda\int|\nb_x^2u_x|^2dx\leq0.
%C(\Omega)||u_x||_{H^1}^2.
 \end{eqnarray}
Hence by the condition $Ric\leq-\lambda<0$ we have
\begin{eqnarray}\label{eq:2.2*}
{d\over dt}\left({-{1\over\lambda}}\int
Ric(u_x,u_x)dx\right)&=&{2\epa\over\lambda}\int
Ric(\nb_x^2u_x,\nb_x^2
u_x)dx\nn\\
&\leq&-2\epa\int|\nb_x^2u_x|^2dx\leq0.
 \end{eqnarray}
 Considering $\int Ric(\nb_xu_x,\nb_xu_x)dx$ we have:
\begin{eqnarray}
&&{}{d\over dt}\int Ric(\nb_xu_x,\nb_xu_x)dx=2\int Ric (\nb_t\nb_xu_x,\nb_xu_x)dx\nn\\
&=&2\int Ric (\nb_x\nb_tu_x,\nb_xu_x)dx+2\int Ric
(R(u_t,u_x)u_x,\nb_xu_x)dx\nn\\
&=&2\int Ric (\nb_x^2u_t,\nb_xu_x)dx+2\int Ric
(u_t,R(\nb_xu_x,u_x)u_x)dx\nn\\
&=&2\int Ric (u_t,\nb^3_xu_x)dx+2\int Ric
(u_t,R(\nb_xu_x,u_x)u_x)dx.\nn
\end{eqnarray}
Thus, substituting (\ref{eq:2.1}) into above we have
\begin{eqnarray}{\label{eq:2.2}}
&&{}{d\over dt}\int Ric(\nb_xu_x,\nb_xu_x)dx \nn\\
&=&-2\epa\int Ric (\nb_x^3u_x,\nb^3_xu_x)dx +2\int Ric
(\nb_x^2u_x,\nb^3_xu_x)dx\nn\\
&&{}+2\ro\int Ric(u_x,u_x)Ric (u_x,\nb^3_xu_x)dx\nn\\
&&{}-2\epa\int Ric (\nb_x^3u_x,R(\nb_xu_x,u_x)u_x)dx
+2\int Ric (\nb_x^2u_x,R(\nb_xu_x,u_x)u_x)dx\nn\\
&&{}+2\ro\int Ric(u_x,u_x)Ric (u_x,R(\nb_xu_x,u_x)u_x)dx.
\end{eqnarray}
It is easy to see that the second term and the last term on the
right hand side vanish since
$$2Ric(\nb_x^2u_x,\nb^3_xu_x)=\nb_x\left(Ric(\nb_x^2u_x,\nb_x^2u_x)\right)$$
and
$$Ric (u_x,R(\nb_xu_x,u_x)u_x)\equiv0.$$
Moreover, for the fifth term on the right, we have
\begin{eqnarray}
&&{}2\int Ric (\nb_x^2u_x,R(\nb_xu_x,u_x)u_x)dx\nn\\
&=&-2\int Ric (\nb_xu_x,(\nb_xR)(\nb_xu_x,u_x)u_x)dx-2\int Ric
(\nb_xu_x,R(\nb^2_xu_x,u_x)u_x)dx,\nn
\end{eqnarray}
which implies that $$2\int Ric
(\nb_x^2u_x,R(\nb_xu_x,u_x)u_x)dx=-\int Ric
(\nb_xu_x,(\nb_xR)(\nb_xu_x,u_x)u_x)dx.$$

For the left terms of (\ref{eq:2.2}), after integration by parts we
get
\begin{eqnarray}
&&{}2\ro\int Ric(u_x,u_x)Ric (u_x,\nb^3_xu_x)dx\nn\\
&=&-2\ro\int Ric(u_x,u_x)Ric (\nb_xu_x,\nb^2_xu_x)dx-4\ro\int Ric(\nb_xu_x,u_x)Ric (u_x,\nb^2_xu_x)dx\nn\\
&=&6\ro\int Ric(\nb_xu_x,\nb_xu_x)Ric (\nb_xu_x,u_x)dx;\nn
\end{eqnarray}
and
\begin{eqnarray}
 &&{}-2\epa\int Ric (\nb_x^3u_x,R(\nb_xu_x,u_x)u_x)dx\nn\\
&=&2\epa\int Ric (\nb_x^2u_x,(\nb_xR)(\nb_xu_x,u_x)u_x)dx+2\epa\int
Ric (\nb_x^2u_x,R(\nb^2_xu_x,u_x)u_x)dx\nn\\&&{}+2\epa\int Ric
(\nb_x^2u_x,R(\nb_xu_x,u_x)\nb_xu_x)dx.\nn
\end{eqnarray}
Hence we obtain that
\begin{eqnarray}{\label{eq:2.3}}
&&{}{d\over dt}\left(-{1\over\lambda}\int Ric(\nb_xu_x,\nb_xu_x)dx\right)-{2\epa\over\lambda}\int Ric (\nb_x^3u_x,\nb^3_xu_x)dx\nn\\
&=&-{1\over\lambda}\left(2\epa\int Ric
(\nb_x^2u_x,(\nb_xR)(\nb_xu_x,u_x)u_x)dx+2\epa\int
Ric (\nb_x^2u_x,R(\nb^2_xu_x,u_x)u_x)dx\right.\nn\\
&&{}+2\epa\int Ric(\nb_x^2u_x,R(\nb_xu_x,u_x)\nb_xu_x)dx-\int Ric
(\nb_xu_x,(\nb_xR)(\nb_xu_x,u_x)u_x)dx\nn\\
&&{}\left.+6\ro\int Ric(\nb_xu_x,\nb_xu_x)Ric (\nb_xu_x,u_x)dx\right)\nn\\
&\leq&
C(\Omega,\ro,\lambda)\left(\int|\nb_x^2u_x||\nb_xu_x||u_x|^3dx+\int|\nb_x^2u_x|^2|u_x|^2dx
+\int|\nb_x^2u_x||\nb_xu_x|^2|u_x|dx\right.\nn\\
&&{}+\left.\int|\nb_xu_x|^2|u_x|^3dx+\int|\nb_xu_x|^3|u_x|dx
\right).
\end{eqnarray}
Utilizing H\"older inequality and the following interpolation
inequalities
\begin{eqnarray}\label{ineq}
%\begin{eqnarray*}
||u_x||_{L^\infty}&\leq& C(\Omega)
\left(||\nabla_x u_x||^2_{L^2}+||u_x||^2_{L^2}\right)^{1\over4}||u_x||^{1\over2}_{L^2};\nn\\
||\nabla_x u_x||_{L^\infty}&\leq& C(\Omega)\left(||\nabla_x^2
u_x||^2_{L^2}+||\nabla_xu_x||^2_{L^2}\right)^{1\over4}||\nabla_xu_x||^{1\over2}_{L^2},
%\end{eqnarray*}
\end{eqnarray}
we obtain that
\begin{eqnarray}{\label{eq:2.4}}
&&{}{d\over dt}\left(-{1\over\lambda}\int
Ric(\nb_xu_x,\nb_xu_x)dx\right)-{2\epa\over\lambda}\int Ric
(\nb_x^3u_x,\nb^3_xu_x)dx\leq
C(\Omega,\ro,\lambda)||u_x||_{H^2}^4.\nn\\
\end{eqnarray}
Now to show (\ref{eq:2.0}), we need compute ${d\over dt}\int
Ric(\nb_x^2u_x,\nb_x^2u_x)dx$ and we have
\begin{eqnarray}\label{eq:2.5}
&&{}{d\over dt}\int Ric(\nabla_x^2u_x,\nabla_x^2u_x)dx=2\int Ric(\nabla_t\nabla_x^2u_x,\nabla_x^2u_x) dx\nn\\
&=&2\int Ric(\nabla_x\nabla_t\nabla_x u_x,\nabla_x^2u_x) dx+2\int Ric( R(u_t,u_x)\nabla_x u_x,\nabla_x^2u_x) dx\nn\\
&=&-2\int Ric(\nabla_x^2u_t,\nabla_x^3u_x) dx-2\int Ric(\nabla_x^3u_x,R(u_t,u_x)u_x) dx\nn\\
&&{}+2\int Ric(\nabla_x^2u_x,R(u_t,u_x)\nabla_x u_x) dx\nn\\
&=&-2\int Ric( u_t,\nabla_x^5u_x) dx-2\int Ric( u_t,R(\nabla_x^3u_x,u_x)u_x) dx\nn\\
&&{}+2\int Ric( u_t,R(\nabla_x^2u_x,\nabla_x u_x)u_x) dx.
\end{eqnarray}
Substituting (\ref{eq:2.1}) into (\ref{eq:2.5}) and noting that
$$\int Ric(\nb_x^3u_x,\nb_x^5u_x)dx=-\int Ric(\nb_x^4u_x,\nb_x^4u_x); \
\int Ric(\nb_x^2u_x,\nb_x^5u_x)dx=0$$
$$\int Ric( u_x,R(\nabla_x^3u_x,u_x)u_x) dx=\int Ric( u_x,R(\nabla_x^2u_x,\nabla_x u_x)u_x)\equiv0,$$
we have
\begin{eqnarray}\label{eq:2.6}
&&{}{d\over dt}\left(-{1\over\lambda}\int Ric(\nabla_x^2u_x,\nabla_x^2u_x)dx\right)-{2\epa\over
\lambda}\int Ric(\nb_x^4u_x,\nb_x^4u_x)dx\nn\\
&=&-{1\over\lambda}\left(2\epa\int Ric(
\nb_x^3u_x,R(\nabla_x^3u_x,u_x)u_x) dx-2\epa\int
Ric( \nb_x^3u_x,R(\nabla_x^2u_x,\nabla_x u_x)u_x) dx\right.\nn\\
&&{}-2\ro\int Ric( Ric(u_x,u_x)u_x,\nabla_x^5u_x) dx-2\int Ric(
\nb_x^2u_x,R(\nabla_x^3u_x,u_x)u_x) dx\nn\\
&&{}\left.+2\int Ric( \nb_x^2u_x,R(\nabla_x^2u_x,\nabla_x u_x)u_x)
dx.\right)
\end{eqnarray}
For the first two terms of (\ref{eq:2.6}) on the right , integrating
by parts yields
\begin{eqnarray}%\label{eq:2.7}
&&{}\quad 2\epa\int  Ric(\nabla_x^3u_x,R(\nabla_x^3u_x,u_x)u_x)
dx-2\epa\int
Ric( \nb_x^3u_x,R(\nabla_x^2u_x,\nabla_x u_x)u_x) dx\nn\\
&=&-2\epa\int  Ric(\nabla_x^3u_x,(\nabla_x R)(\nabla_x^2u_x,u_x)u_x) dx-2\epa\int  Ric(\nabla_x^4u_x,R(\nabla_x^2u_x,u_x)u_x) dx\nn\\
&&{}-2\epa\int  Ric(\nabla_x^3u_x,R(\nabla_x^2u_x,u_x)\nabla_x u_x)
dx\nn\\
&&{}-4\epa\int  Ric(\nabla_x^3u_x,R(\nabla_x^2u_x,\nabla_x u_x)u_x)
dx.
\end{eqnarray}
Hence for any $\delta>0$,
\begin{eqnarray}%\label{eq:2.7}
&&{}{-{1\over\ld}}\left( 2\epa\int
Ric(\nabla_x^3u_x,R(\nabla_x^3u_x,u_x)u_x)dx-2\epa\int
Ric( \nb_x^3u_x,R(\nabla_x^2u_x,\nabla_x u_x)u_x) dx\right)\nn\\
&\leq&{\epa\delta\over\ld} \int|\nabla_x^4u_x|^2dx+{4\epa\delta\over\ld} \int|\nabla_x^3u_x|^2dx\nn\\
&&{}+{\epa C(\Omega)\over 2\delta\ld}\left\{\int|(\nabla_x
R)(\nabla_x^2u_x,u_x)u_x|^2dx
+\int|R(\nabla_x^2u_x,u_x)u_x|^2dx\right.\nn\\
&&{}+\left.\int|R(\nabla_x^2u_x,u_x)\nabla_x u_x|^2dx
+2\int|R(\nabla_x^2u_x,\nabla_x u_x)u_x|^2dx\right\}\nn\\
&\leq&{\epa\delta\over\ld}
\int|\nabla_x^4u_x|^2dx+{4\epa\delta\over\ld}
\int|\nabla_x^3u_x|^2dx\nn\\
&&{}+{C(\Omega)\over2\delta\ld}\int|\nabla_x^2u_x|^2(|u_x|^6+|u_x|^4+|\nabla_x u_x|^2|u_x|^2)dx\nn\\
&\leq&{\epa\delta\over\ld}
\int|\nabla_x^4u_x|^2dx+{4\epa\delta\over\ld}
\int|\nabla_x^3u_x|^2dx+{C(\Omega)\over2\delta\ld}(||u_x||^4_{H^2}+||u_x||^6_{H^2}+||u_x||^8_{H^2})\nn\\
&\leq&-{\epa\delta\over \lambda^2} \int
Ric(\nabla_x^4u_x,\nabla_x^4u_x)dx-{4\epa\delta\over \lambda^2} \int
Ric(\nabla_x^3u_x,\nabla_x^3u_x)dx\nn\\
&&{}+{C(\Omega)\over2\delta\ld}\left(||u_x||^4_{H^2}+||u_x||^6_{H^2}+||u_x||^8_{H^2}\right).
\end{eqnarray}
For the third term of (\ref{eq:2.6}), integrating by parts yields
\begin{eqnarray}\label{eq:2.22}
&&{}-\int Ric( Ric(u_x,u_x)u_x,\nabla_x^5u_x) dx\nn\\
&=&\int Ric( \nabla_x^4u_x,\nb_xu_x) Ric(u_x,u_x)dx+2\int Ric(
\nabla_x^4u_x,u_x) Ric(\nb_xu_x,u_x)dx\nn\\
&=&-\int Ric( \nabla_x^3u_x,\nb^2_xu_x) Ric(u_x,u_x)dx-4\int Ric(
\nabla_x^3u_x,\nb_xu_x) Ric(\nb_xu_x,u_x)dx\nn\\
&&{}-2\int Ric( \nabla_x^3u_x,u_x) Ric(\nb_x^2u_x,u_x)dx-2\int Ric(
\nabla_x^3u_x,u_x) Ric(\nb_xu_x,\nb_xu_x)dx\nn\\
&=&5\int Ric( \nabla_x^2u_x,\nb^2_xu_x) Ric(\nb_xu_x,u_x)dx+10\int
Ric( \nabla_x^2u_x,\nb_xu_x) Ric(\nb^2_xu_x,u_x)dx\nn\\
%&&{}+6\int Ric( \nabla_x^2u_x,\nb_xu_x) Ric(\nb_xu_x,\nb_xu_x)dx
&\leq&C(\Omega)\int |\nb_x^2u_x|^2|\nb_xu_x||u_x|dx\leq
C(\Omega)||\nb_xu_x||_{L^\infty}||u_x||_{L^\infty}\int
|\nb_x^2u_x|^2dx.
\end{eqnarray}
Thus by (\ref{ineq}) we have
\begin{eqnarray} \label{eq:2.7}
&&{}{2\ro\over\ld}\int Ric( Ric(u_x,u_x)u_x,\nabla_x^5u_x) dx\leq
C(\Omega,\ro,\ld)||u_x||^4_{H^2}.
\end{eqnarray}
It is easy to check that the other two terms of (\ref{eq:2.6}) are
also bounded by by $C(\Omega) ||u_x||^4_{H^2}$ via the similar
argument, we omit the detail. This together with
Ineq.(\ref{eq:2.2*}), (\ref{eq:2.4}) and
(\ref{eq:2.6})-(\ref{eq:2.7}) yields
\begin{eqnarray}\label{eq:2.8}
0&<&{d\over dt}\int
(-{1\over\ld})\Big\{Ric(\nabla_x^2u_x,\nabla_x^2u_x)+Ric(\nb_xu_x,\nb_xu_x)+Ric(u_x,u_x)\Big\}dx\nn\\
&&{}-\left({2\epa\over\ld}-{\epa\delta\over \lambda^2} \right)\int
Ric(\nb_x^4u_x,\nb_x^4u_x)dx
-\left({2\epa\over\ld}-{4\epa\delta\over
\lambda^2}\right)\int Ric(\nabla_x^3u_x,\nabla_x^3u_x)dx\nn\\
&\leq&{C(\Omega,\ro,\ld)}\left({1\over2\delta}+1\right)\Big(||u_x||^4_{H^2}
+||u_x||^6_{H^2}+||u_x||^8_{H^2}\Big).
\end{eqnarray}
Thus, let $\delta={\lambda\over8}$, we have
\begin{eqnarray}\label{eq:2.8}
0&<&{d\over dt}\left(-{1\over\ld}\sum_{s=0}^2\int
Ric(\nabla_x^su_x,\nabla_x^su_x)dx\right)\leq{C(\Omega,\ro,\lambda)}\sum_{l=2}^4||u_x||^{2l}_{H^2}.
\end{eqnarray}
Furthermore, by the assumption that the Ricci curvature on $N$ has a
negative upper bound $-\lambda$, we have
$$||u_x||_{H^2}^2=\sum_{s=0}^2\int |\nb_x^su_x|^2dx\leq -{1\over \lambda}\sum_{s=0}^2\int Ric(\nb_x^su_x,\nb_x^su_x)dx.$$
Hence from (\ref{eq:2.8}) we could obtain that
\begin{eqnarray}\label{eq:2.9}
{d\over dt}\left(-{1\over\ld}\sum_{s=0}^2\int
Ric(\nabla_x^su_x,\nabla_x^su_x)dx\right)
\leq{C(\Omega,\ro,\lambda)}\sum_{l=2}^4\left(-{1\over\ld}\sum_{s=0}^2\int
Ric(\nb_x^su_x,\nb_x^su_x)dx\right)^l.
\end{eqnarray}

If $k\geq4$, then for $3\leq m\leq k-1$, by the similar argument, we
could get
\begin{eqnarray}\label{eq:2.10}
{d\over dt}\left(-{1\over\ld}\sum_{s=0}^m\int
Ric(\nabla_x^su_x,\nabla_x^su_x)dx\right) \leq
C(\Omega,\ro,\lambda,Q_{m-1})\left(-{1\over\ld}\right)\sum_{s=0}^m\int
Ric(\nabla_x^su_x,\nabla_x^su_x)dx,\nn\\
\end{eqnarray}
where $$Q_{m-1}(u)=-{1\over\ld}\sum_{i=0}^{m-1}\int
Ric(\nb_x^iu_x,\nb_x^iu_x)dx\geq0,$$ $C(\Omega,\ro, \lambda,
Q_{m-1})$ only depends on $Q_{{m-1}}$, $\ro$, $\lambda$, the bounds
on the Ricci curvature and the bounds on the curvature $R$ and its
covariant derivatives $\nabla^lR$ with $l\leq m$ on $\Omega\subset
N$. We omit the details of the proof. We should note that by the
definition of $Q_{{m-1}}(u)$ we have $$Q_{{m-1}}(u)\leq C(\Omega,\ld
)\sum_{i=0}^{m-1}\int
|\nb_x^iu_x|^2dx=C(\Omega)||u_x||^2_{H^{m-1}}.$$

Thus, if we let $$f(t)=-{1\over\ld}\sum_{s=0}^2\int
Ric(\nabla_x^su_x,\nabla_x^su_x)dx+1,$$ then we have
\begin{eqnarray}\label{eq:2.11}
{df\over dt}\leq C(\Omega)f^4,\quad
f(0)=-{1\over\ld}\sum_{s=0}^2\int
Ric(\nabla_x^su_{0x},\nabla_x^su_{0x})dx+1.
\end{eqnarray}
It follows from (\ref{eq:2.11}) that there exists constants $T_0>0$
and $C_0>0$ such that
\begin{eqnarray*}
||u_{x}||^2_{H^2}\leq -{1\over\lambda}\sum_{s=0}^2\int
Ric(\nabla_x^su_x,\nabla_x^su_x)dx\leq C_0,\quad
t\in[0,\text{min}(T_0,T')].
\end{eqnarray*}
Now let $T=\text{min}(T_0,T')$. If $m=3$, by the Gronwall
inequality, we can obtain from (\ref{eq:2.10}):
\begin{eqnarray*}
||u_{ x}||_{H^3}^2\leq-{1\over\lambda}\sum_{s=0}^3\int
Ric(\nabla_x^su_x,\nabla_x^su_x)dx\leq  C_1\left(\Omega,\lambda,
%\sum_{s=0}^3\int Ric(\nabla_x^su_x,\nabla_x^su_x)dx \lambda,
T, Q_3(u_0)\right),\quad\text{for all}\quad t\in[0,T].
%\sum_{s=0}^m\int Ric(\nabla_x^su_{0x},\nabla_x^su_{0x})dx\right),
\end{eqnarray*}
Then by induction we have that there exists a constant
$C_{m-2}(\Omega, \lambda, Q_m(u_0))>0$, such that for any $3\leq
m\leq k-1$
\begin{eqnarray}\label{eq:2.12}
\text{ess sup}_{t\in[0,T]}||u_x||^2_{H^m}\leq C_{m-2}\left(\Omega,
\lambda, Q_m(u_0)\right)\leq C_{m-2}\left(\Omega, \lambda,
||u_{0x}||_{H^{m}}\right).
\end{eqnarray}

If $N$ is of uniform bounds on the curvature tensor and its
derivatives $\nabla^lR$ with $l\leq m$, it is easy to see from the
above arguments that $T=T_0$ since the coefficients of the above
differential inequalities depend only on the bounds on Ricci
curvature, the Riemann curvature tensor $R$ and its covariant
derivatives $\nabla^l R$ of order $l\leq m$ on $N$. That is
$T=T(N,Q_3(u_0),\lambda)$ depends only on $N$, $u_0$ and $\lambda$,
not on $0<\epa<1$.

Now we consider the case $N$ is a noncompact, complete Riemannian
manifold with parallel Ricci tensor and the Ricci curvature has a
negative upper bound $-\lambda<0$. Note that a positive lower bound
of $T'$ can also be derived from (\ref{eq:2.12}) when $k\geq 5$.
Indeed, It is easy to see from the approximate equation
(\ref{eq:2.1}) and the interpolation inequalities (see Theorem 2.1
in \cite{YD} for details) that (\ref{eq:2.12}) implies
$$\text{ess sup}_{t\in[0,T]}||u_t||_{L^2(S^1, TN)}\leq C(\Omega,\lambda, Q_3(u_0))\leq C(\Omega,\lambda,||u_{0x}||_{H^3}).$$
On the other hand, from the approximate equation of the geometric
flow (\ref{eq:2.1}) we have
$$\nabla_xu_t=-\epa\nabla_x^4u_x +\nabla_x^3u_x + \ro \nabla_x(Ric(u_x, u_x)u_x).$$
Hence, when $k\geq5$ we infer from (\ref{eq:2.12}) and the
interpolation inequality that
$$\text{ess sup}_{t\in[0,T]}||u_t||_{H^1(S^1, TN)}\leq C(\Omega, \|u_{0x}\|_{H^4}).$$

Moreover, for some $0<a<1$ the following interpolation inequality
holds
\begin{eqnarray*}
||u_t(s)||_{L^\infty}\leq C||u_t(s)||^a_{H^1}||u_t(s)||^{1-a}_{L^2}.
\end{eqnarray*}
This implies that, for some $\mathcal{M}>0$, there holds true
$$\text{ess sup}_{t\in[0,T]}||u_t||_{L^\infty}\leq \mathcal{M}.$$
Thus we have
\begin{eqnarray*}
\sup_{x\in S^1} d_N(u(x,t), u_0(x)) \leq \mathcal{M}t, \quad \text
{for}\quad t<T.
\end{eqnarray*}
If $T'>T_0$ we get the lower bound, so we may assume that $T'\leq
T_0$. Then letting $t\rightarrow T'$ in the above inequality we get
$\mathcal{M} T'\geq1$. Therefore, if we set
$T=\text{min}\{{1\over\mathcal{M}}, T_0\},$ then the desired
estimates hold for $t\in[0,T].$

It is easy to find that the solution to (\ref{eq:2.1}) with
$\epa\in(0,1)$ must exist on the time interval $[0,T]$. Otherwise,
we always extend the time interval of existence to cover $[0,T]$.
Hence we always have $T_\epa\geq T$ and then we complete the proof
of this lemma.
\end{proof}

Here we should point out that if the ricci curvature of a complete
Riemannian manifold $N$ has a positive lower bound, i.e.
$Ric\geq\ld>0$, then by Myers's theorem $N$ must be compact. Hence
in this case we have the following corollary via easier arguments
than that in Lemma \ref{lm:2.1}.
\begin{coro}\label{lm:2.1*}
Assume that $N$ is a complete Riemannian manifold with parallel
Ricci tensor, (i.e. $\nb Ric\equiv0$). If $N$ has positive lower
bounds on the Ricci curvature ($Ric\geq \lambda>0$) and $u_0\in
H^k(S^1,N)$ with an integer $k\geq3$. Then there exists a constant
$T=T(||u_0||_{H^3})$, independent of $\epa\in(0,1)$, such that if
$u\in C([0,T_\epa),H^k(S^1,N))$ is a solution of (\ref{eq:2.1}) with
$\epa\in(0,1)$, then $T(||u_0||_{H^3})\leq T_\epa$ and
$||u(t)||_{H^{m+1}}\leq C(||u_0||_{H^{m+1}})$ for any integer $2\leq
m\leq k-1$.
\end{coro}
We omit the details of the proof since the process is similar and in
this case $N$ is compact. Here we point out that instead of showing
(\ref{eq:2.0}), it is suffices to show
\begin{eqnarray}
{d\over dt}\left(\sum_{s=0}^2\int
Ric(\nabla_x^su_x,\nabla_x^su_x)dx\right)\leq{C(\Omega,\lambda)}\sum_{l=2}^4\left(\sum_{s=0}^2\int
Ric(\nb_x^su_x,\nb_x^su_x)dx\right)^l.\nn
\end{eqnarray}

Following Lemma\ref{lm:2.1}, we could obtain the following local
existence results of the Cauchy problem (\ref{eq:0.1}).
\begin{lem}\label{lem:2.2}
Let $(N, h)$ is a complete Riemannian manifold with parallel Ricci
tensor. If the Ricci curvature has a negative upper bound
$-\lambda<0$ and the curvature tensor $R$ and its covariant
derivatives of any order have uniform bounds $($i.e.,
$|\nabla^lR|\leq B_l$, $l=0, 1, 2, \cdots$$)$, then, for any integer
$k\geq3$ the Cauchy problem of (\ref{eq:0.1}) with the initial value
map $u_0\in H^k(S^1, N)$ admits a local solution $u\in
L^\infty([0,T],H^k(S^1, N))$, where $T=T(N,||u_0||_{H^3})$.

\end{lem}

Before proving Lemma \ref{lem:2.2}, we remark that in \cite{YD},
Ding and Wang have shown that the $H^{m}$ norm of section $\nabla u$
is equivalent to the usual Sobolev $W^{m+1, 2}$ norm of the map $u$.
Precisely, we have
\begin{lem}\label{lm:w} (\cite{YD})
Assume that $N$ is a compact Riemannian manifold with or without
boundary and $m\geq 1$. Then there exists a constant $C = C(N,m)$
such that for all $u \in C^\infty(S^1,N)$,
\[ \Norm{Du}_{W^{m-1,2}} \leq C\sum_{i=1}^m \Norm{\nabla u}^i_{H^{m-1,2}} \]
and
\[ \Norm{\nabla u}_{H^{m-1,2}} \leq C\sum_{i=1}^m \Norm{Du}^i_{W^{m-1,2}}. \]\\
\end{lem}
Now we turn to the proof of Lemma \ref{lem:2.2}. The process goes
similar with the proof of Lemma 3.3 in (\cite{SW}) which is about
the local existence of the KdV geometric flow .
\begin{proof}
Assume $N$ is compact and we imbed $N$ into $\mathbb{R}^n$. If
$u_0:S^1\rightarrow N$ is $C^\infty$, then from Lemma \ref{lm:2.1}
we have that the Cauchy problem (\ref{eq:2.1}) admits a smooth
solution $u_\epa$ which satisfies the estimates in Lemma
\ref{lm:2.1}. Hence by Lemma \ref{lm:2.1} and Lemma \ref{lm:w},  for
any integer $p>0$ and $\epa\in(0,1]$ we have:
\begin{eqnarray}\label{eq:ww}
\sup_{t\in[0,T]} ||u_\epa||_{W^{p,2}(N)}\leq C_p(N,u_0),
\end{eqnarray}
where $C_p(N,u_0)$ does not depend on $\epa$. Hence, by sending
$\epa\rightarrow 0$ and applying the embedding theorem of Sobolev
spaces to $u$, we have $u_\epa\rightarrow u\in C^p(S^1\times[0,T])$
for any $p$. It is easy to check that $u$ is a solution to the
Cauchy problem (\ref{eq:0.1}).

If $u_0:S^1\rightarrow N$ is not $C^\infty$, but $u_0\in
W^{k,2}(S^1,N)$, we may always select a sequence of $C^\infty$ maps
 $u_{i0}:S^1\rightarrow N$, $(i=1,2,\cdots,n\cdots)$, such that
$$u_{i0}\rightarrow u_0\quad\text{in}\quad W^{k,2},\quad\text{as}\quad i\rightarrow \infty.$$
Thus following from Lemma \ref{lm:w} we have
\begin{eqnarray*}
||\nabla_x u_{i0}||_{H^{k-1}}\rightarrow ||\nabla_x
u_0||_{H^{k-1}},\quad\text{as}\quad i\rightarrow\infty.
\end{eqnarray*}
Thus there exists a unique, smooth solution $u_i$, defined on time
interval $[0,T_i]$, of the Cauchy problem (\ref{eq:2.1}) with $u_0$
replaced by $u_{i0}$. Furthermore, from Lemma \ref{lm:2.1} we could
obtain that if $i$ is large enough, then there exists a uniform
positive lower bound of $T_i$, denoted by $T$, such that the
following inequality holds uniformly with respect to large enough
$i$:
\begin{eqnarray*}
\sup_{t\in[0,T]} ||\nabla u_i(t)||_{H^{k-1}}\leq
C(T,||u_{0x}||_{H^{k-1}}).
\end{eqnarray*}
Hence from Lemma \ref{lm:w} we deduce
\begin{eqnarray}\label{eq:www}
\sup_{t\in[0,T]} ||Du_i(t)||_{W^{k-1,2}}\leq
C(T,||u_{0x}||_{W^{k-1,2}}),
\end{eqnarray}
and by (\ref{eq:2.1}) we have
$${du_i\over dt}\in L^2([0,T],W^{k-3,2}(S^1,N)).$$
By Sobolev theorem, it is easy to see that $u_i\in
C^{0,{1\over2}}([0,T],W^{k-3,2}(S^1,N)).$

Interpolating the spaces $L^\infty([0,T],W^{k,2}(S^1,N)) $ and
$C^{0,{1\over2}}([0,T],W^{k-3,2}(S^1,N))$ yields that
\begin{equation}\label{eq:vv}
u_i\in
C^{0,\gamma}([0,T],W^{k-6\gamma,2}(S^1,N))\quad\text{for}\quad
\gamma\in(0,{1\over2}).
\end{equation}
Therefore when letting $\gamma$ small while using Rellich's theorem
and the Ascoli-Arzela theorem,
 from (\ref{eq:www}) and (\ref{eq:vv}) we obtain that there exists $$u\in
L^\infty([0,T],W^{k,2}(S^1,N))\cap C([0,T],W^{k-1,2}(S^1,N))$$ such
that
\begin{eqnarray*}
u_i&\rightarrow& u\quad[\text{weakly}^*]\quad\text{in}\quad
L^\infty([0,T],W^{k,2}(S^1,N)),\nn\\
u_i&\rightarrow& u\quad\text{in}\quad C([0,T],W^{k-1,2}(S^1,N))
\end{eqnarray*}
upon extracting a subsequence and re-indexing if necessary.

It remains to verify that $u$ is a strong solution to
(\ref{eq:0.1}). We need to check that for any $v\in
C^\infty(S^1\times[0,T],\mathbb{R}^n)$ there holds
\begin{eqnarray*}
\int_0^T\int_{S^1}\la u_t,v\ra dx dt&=&\int_0^T\int_{S^1}\la
\nabla^2_xu_x,v\ra dx dt + {\ro}\int_0^T\int_{S^1}\la
Ric(u)(u_x,u_x)u_x,v\ra dx dt.
\end{eqnarray*}

First we always have that for each $u_i$
\begin{eqnarray*}
\int_0^T\int_{S^1}\la u_{it},v\ra dx dt&=&\int_0^T\int_{S^1}\la
\nabla^2_xu_{ix},v\ra dx dt + {\ro}\int_0^T\int_{S^1}\la
Ric(u_i)(u_{ix},u_{ix})u_{ix},v\ra dx dt..
\end{eqnarray*}
For each $y\in N\subset\mathbb{R}^n$, let $P(y)$ be the orthogonal
projection from $\mathbb{R}^n$ onto $T_yN$, we have
\begin{eqnarray}
\nabla_x u_x&=&P(u)u_{xx},\nn\\
\nabla^2_xu_x&=&P(u)(P(u))_xu_{xx}+P(u)u_{xxx},\label{eq:2.23}\\
\nabla^2_xu_{ix}&=&P(u_i)(P(u_i))_xu_{ixx}+P(u_i)u_{ixxx}.\label{eq:2.24}
\end{eqnarray}

Hence we have
\begin{eqnarray}\label{eq:2.26*}
&&{}\int_0^T\int_{S^1}|\la\nabla^2_xu_x,v\ra- \la\nabla^2_xu_{ix},v\ra| dx dt\nn\\
%&\leq&\int_0^T\int_{S^1}|\la P(u)u_{xxx}-P(u_i)u_{ixxx}, v\ra|dx dt\nn\\
%&&{}+ \int_0^T\int_{S^1}|\la P(u)(P(u))_xu_{xx}-P(u_i)(P(u_i))_xu_{ixx},v\ra|dx dt\nn\\
&\leq&\int_0^T\int_{S^1}|\la\big( P(u)-P(u_i)\big)u_{xxx}, v\ra|dx dt+\int_0^T\int_{S^1}|\la P(u_i)\big(u_{xxx}-u_{ixxx}\big), v\ra|dx dt\nn\\
&&{}+ \int_0^T\int_{S^1}|\la \big(P(u)(P(u))_x-P(u_i)(P(u_i))\big)_xu_{xx},v\ra|dx dt\nn\\
&&{}+ \int_0^T\int_{S^1}|\la
P(u_i)(P(u_i))_x\big(u_{xx}-u_{ixx}\big),v\ra|dx dt.
\end{eqnarray}
Moreover,
\begin{eqnarray}\label{eq:2.25}
&&{}\int_0^T\int_{S^1}\left|\la Ric(u)(u_x,u_x)u_x,v\ra-\la Ric({u_i})(u_{ix},u_{ix})u_{ix},v\ra\right| dx dt\nn\\
&\leq&\int_0^T\int_{S^1}|
Ric(u)(u_x,u_x)u_x-Ric({u_i})(u_{ix},u_{ix})u_{ix}||v|dx dt.
\end{eqnarray}
Since $N$ is compact, it is obviously that
$$ ||P(\cdot)D(P(\cdot))||_{L^\infty(N)}<\infty|;\quad ||Ric(\cdot)||_{L^\infty(N)}<\infty.$$
Hence we obtain that each term on the right hand side of
(\ref{eq:2.26*}) and (\ref{eq:2.25}) converges zero as $i$ goes to
infinity. This implies that
\begin{eqnarray*}
%&&{}\lim_{i\rightarrow \infty}\int_0^T\int_{S^1}\la J_{u_i}\nabla_xu_{ix},v\ra dx dt=\int_0^T\int_{S^1}\la J_u\nabla_xu_x,v\ra dx dt;\\
&&{}\lim_{i\rightarrow \infty}\int_0^T\int_{S^1}\la \nabla^2_xu_{ix},v\ra dx dt=\int_0^T\int_{S^1}\la \nabla^2_xu_x,v\ra dx dt;\\
&&{}\lim_{i\rightarrow \infty}\int_0^T\int_{S^1}\la
Ric({u_i})(u_{ix},u_{ix})u_{ix},v\ra dx dt=\int_0^T\int_{S^1}\la
Ric(u)(u_x,u_x)u_x,v\ra dx dt.
\end{eqnarray*}
On the other hand, we also have
\begin{eqnarray*}
\lim_{i\rightarrow \infty}\int_0^T\int_{S^1}\la u_{it},v\ra dx
dt=-\int _0^T\int_{S^1}\la u,v_t\ra dx dt +\int_{S^1}(\la
u(T),v(T)\ra-\la u_0,v(0)\ra) dx .
\end{eqnarray*}
Thus, from the above equalities we have
\begin{eqnarray}\label{eq:2.26}
&&{} \int_0^T\int_{S^1}\la \nabla^2_xu_x,v\ra dx dt +
{\ro}\int_0^T\int_{S^1}\la Ric(u_x,u_x)u_x,v\ra dx dt\nn\\
&=&-\int _0^T\int_{S^1}\la u,v_t\ra dx dt+\int_{S^1}(\la
u(T),v(T)\ra-\la u_0,v(0)\ra) dx .
\end{eqnarray}
Note that $\nabla^2_x u_x\in L^2(S^1\times[0,T], \mathbb{R}^n)$,
thus (\ref{eq:2.26}) implies $u_t\in
L^2(S^1\times[0,T],\mathbb{R}^n)$. Therefore for any smooth function
$v$ we always have
\begin{eqnarray*}
\int_0^T\int_{S^1}\la u_t,v\ra dx dt&=&\int_0^T\int_{S^1}\la
\nabla^2_xu_x + {\ro}Ric(u_x,u_x)u_x,v\ra dx dt,
\end{eqnarray*}
which means that $u$ is a strong solution of (\ref{eq:2.1}).

It is easy to see that if $N$ is a noncompact manifold with bounded
geometry and the domain is $S^1$, we could find a compact subset of
$N$, denoted by $\Omega$, such that $u_0(S^1)\subset \Omega\subset
\mathbb{R}^n$. Therefore we could repeat the same process as in the
case $N$ is compact then we obtain the same results and complete the
proof.
\end{proof}
Now we could show the uniqueness of the solutions and prove Theorem
{\ref{thm:1.1}}.

\noindent{\bf\textit{Proof of Theorem \ref{thm:1.1}}}. Without loss
of generality, we always assume that $N$ is compact, since $u(x,
t)\in L^\infty([0, T], H^4(S^1, N))$ implies that $\{u(x, t): (x, t)
\in S^1\times[0, T]\}\subset\subset N$. We regard $N$ as a
submanifold of $\mathbb{R}^n$. Let $u$, $v:S^1\times[0,T]\rightarrow
N\subset\mathbb{R}^n$ be two solutions of (\ref{eq:0.1}) such that
$u(x,0)=v(x,0)=u_0$ and $u,v\in L^\infty([0,T],W^{k,2}(S^1,N))$ for
$k\geq4$. Let $w=u-v$ which makes sense as a $\mathbb{R}^n$-valued
function. It is worthy to point out that the Ricci curvature  $Ric$
here should be regarded as operators on $\mathbb{R}^n$, such that
$Ric(u)(u_x,u_x)u_x-Ric(v)(v_x,v_x)v_x$ makes sense in
$\mathbb{R}^n$.

From (\ref{eq:2.23}) we have that
\begin{eqnarray}\label{eq:2.27}
\nabla_x^2u_x=P(u)u_{xxx}+P(u)(P(u))_xu_{xx}.\nn
\end{eqnarray}
Thus
\begin{eqnarray*}
u_t=P(u)u_{xxx}+P(u)(P(u))_xu_{xx}+{\ro}Ric(u)(u_x,u_x)u_x.
\end{eqnarray*}
Hence we have
\begin{eqnarray}\label{eq:2.28}
w_t&=&P(u)w_{xxx}+[P(u)-P(v)]v_{xxx}\nn\\
&&+P(u)(P(u))_xw_{xx}+\Big(P(u)(P(u))_x-P(v)(P(v))_x\Big)v_{xx}\nn\\
&&{}+{\ro}\Big(Ric(u)(u_x,u_x)u_x-Ric(v)(v_x,v_x)v_x\Big).
\end{eqnarray}

We could proove that there exists a constant $C$ which depends only
on $N$ and $||u||_{W^{4,2}}$ and $||v||_{W^{4,2}}$ such that
\begin{eqnarray}\label{eq:2.29}
{d\over dt}||w||^2_{W^{1,2}}\leq C ||w||^2_{W^{1,2}},
\end{eqnarray}
then by Gronwall's inequality we could obtain that $w\equiv0$ and
obtain the uniqueness of the solutions. To see this, we start by
calculating
\begin{eqnarray}\label{eq:2.30}
&&{}{1\over 2}{d\over dt}\int |w_x|^2dx=-\int \la w_{xx},w_t\ra dx\nn\\
&=&-\int\la w_{xx}, P(u)w_{xxx}\ra dx-\int\la w_{xx},[P(u)-P(v)]v_{xxx}\ra dx\nn\\
&&{}-\int\la w_{xx},P(u)(P(u))_xw_{xx}\ra dx-\int\la w_{xx},[P(u)(P(u))_x-P(v)(P(v))_x]v_{xx}\ra dx\nn\\
&&{}-{\ro}\int\la
w_{xx},\big(Ric(u)(u_x,u_x)u_x-Ric(v)(v_x,v_x)v_x\big)\ra
dx.\end{eqnarray} Then, similar with the process in \cite{SW}, for
the first four terms of ({\ref{eq:2.30}}), we have :
\begin{eqnarray}\label{eq:2.31}
&&{}-\int\la w_{xx}, P(u)w_{xxx}\ra dx-\int\la w_{xx},[P(u)-P(v)]v_{xxx}\ra dx\nn\\
&&{}-\int\la w_{xx},P(u)(P(u))_xw_{xx}\ra dx-\int\la w_{xx},[P(u)(P(u))_x-P(v)(P(v))_x]v_{xx}\ra dx\nn\\
&\leq& C
\left(\int|w_x|^2dx+\int|w_x||w|dx+||w||_{L^\infty}\int(|w_x|+|w|)|v_{xxx}|dx
+\int|w_x|^2|v_{xxx}|dx\right)\nn\\
%&&{}||w||_{L^\infty}\int(|w_x|+|w|)|v_{xxx}|dx\nn\\
&\leq&C||w||^2_{W^{1,2}},
\end{eqnarray}
where $C$ depends on $N$, $||u||_{W^{3,2}}$ and $||v||_{W^{4,2}}$.
The calculations about these estimates are same with that in
\cite{SW} and we omit the details here. For the last term of
(\ref{eq:2.30}), we have

\begin{eqnarray}\label{eq:2.32}
&&{}-{1\over2}\int\la
w_{xx},\big(Ric(u)(u_x,u_x)u_x-Ric(v)(v_x,v_x)v_x\big)\ra dx\nn\\
&=&-{1\over2}\int Ric(u)(u_x,u_x)\la w_{xx},w_x\ra dx\nn\\
&&{}-{1\over2}\int \big(Ric(u)(u_x,u_x)-Ric(v)(v_x,v_x)\big)\la
w_{xx},v_x\ra dx\nn\\
&=&{1\over4}\int D_x(Ric(u)(u_x,u_x))\la w_{x},w_x\ra
dx\nn\\
&&{}+{1\over2}\int \big(Ric(u)(u_x,u_x)-Ric(v)(v_x,v_x)\big)_x\la
w_{x},v_x\ra dx\nn\\
&&{}+{1\over2}\int \big(Ric(u)(u_x,u_x)-Ric(v)(v_x,v_x)\big)\la
w_{x},v_{xx}\ra dx\nn\\
&\leq&C\left(\int |w_x|^2dx+\int
|w_x||v_x|dx+\int|w_x||v_{xx}|dx\right)\nn\\
&\leq& C||w||_{W^{1,2}}^2,
\end{eqnarray}
where $C$ depends on $N$, $||u||_{W^{3,2}}$ and $||v||_{W^{3,2}}$.

Combining (\ref{eq:2.31}) and (\ref{eq:2.32}) yields
\begin{eqnarray}\label{eq:2.33}
{d\over dt}\int |w_x|^2dx&\leq& C||w||_{W^{1,2}}^2,
\end{eqnarray}
where $C$ depends on $N$, $||u||_{W^{3,2}}$ and $||v||_{W^{4,2}}$.

Moreover, by a similar argument we could obtain that
\begin{eqnarray}\label{eq:2.34}
{d\over dt}\int |w|^2dx\leq C||w||^2_{W^{1,2}},
\end{eqnarray}
where $C$ depends on $N$, $||u||_{W^{3,2}}$ and $||v||_{W^{3,2}}$.
We omit the detail.

Hence we have
\begin{eqnarray}
{d\over dt}||w||^2_{W^{1,2}}\leq C||w||^2_{W^{1,2}},\nn
\end{eqnarray}
where $C$ depends on $N$, $||u||_{W^{4,2}}$ and $||v||_{W^{4,2}}$.
This implies that $w\equiv0$ since $w(x,0)=0$, i.e. the solution is
unique.

Thus it suffices to show that$\nabla^{k-1}_xu_x\in
C([0,T];L^2(S^1,TN))$ for $k\geq4$. In the proof of Lemma
\ref{lem:2.2} we have seen that the solution $u\in
L^\infty([0,T],H^{k}(S^1,N))\cap C([0,T],H^{k-1}(S^1,N)),$ thus by
the discussion about (\ref{eq:ww}), (\ref{eq:www}) and the equation
of the new geometric flow, we could easily get that
$${d\over dt}||\nabla^{k-1}_xu_x||^2_{L^2}\leq C,$$
which implies that
$$||\nabla_x^{k-1}u_x(t,x)||^2_{L^2(S^1,TN)}\leq ||\nabla^{k-1}_xu_x(0,x)||^2_{L^2(S^1,TN)}+Ct.$$

Hence we obtain
$$\lim_{t\rightarrow 0}\sup||\nabla_x^{k-1}u_x(t,x)||^2_{L^2(S^1,TN)}\leq ||\nabla^{k-1}_xu_x(0,x)||^2_{L^2(S^1,TN)}.$$

On the other hand, $u\in L^\infty([0,T],H^{k}(S^1,N))\cap
C([0,T],H^{k-1}(S^1,N))$ implies that, with respect to $t$,
$\nabla^{k-1}_xu_x(t,x)$ is weakly continuous in $L^2(S^1,TN)$, we
have
$$||\nabla^{k-1}_xu_x(0,x)||^2_{L^2(S^1,TN)}\leq \lim_{t\rightarrow0}\inf||\nabla^{k-1}_xu_x(t,x)||^2_{L^2(S^1,TN)}.$$

%$$£¨\partial^k u(0,x),\phi(x))=\inf\lim_{t\rightarrow0}(\partial^k u(t,x),\phi(x))< \inf\lim_{t\rightarrow0}|\partial^k u(t,x)||\phi(x)|$$

Thus,
$$\lim_{t\rightarrow0} ||\nabla^{k-1}_xu_x(t,x)||^2_{L^2}=||\nabla^{k-1}_xu_x(0,x)||^2_{L^2},$$
which implies that $\nabla^{k-1}_xu_x(t,x)$ is continuous in
$L^2(S^1,TN)$ at $t=0$. Now by the uniqueness of $u(t,x)$, we get
that $\nabla_x^{k-1}u_x(t,x)$ is continuous at each $t\in[0,T]$,
i.e. $u\in C([0,T],H^k(S^1,N))$ for all $k\geq4$. Thus we complete
the proof of Theorem \ref{thm:1.1}. However, if $k\leq3$, we could
not get the continuity of $||u||_{H^k}$
about $t$ on $[0,T]$ without the uniqueness of the solutions.  \hspace*{\fill}$\Box$\\

We are now ready to proof Theorem \ref{thm:1.2}.
\begin{proof} We only discuss the case that $Ric\leq-\ld<0$ here. For $Ric\geq\ld>0$, the process
is easier. To show the existence of the Cauchy problem
(\ref{eq:0.1}) with an initial map $u_0\in H^4(S^1, N)$, we first
consider the following Cauchy problems:
\begin{equation}\label{eq:2.35}
\left\{
\begin{aligned}
&u_t=
\nabla_x^2u_x+{\ro}Ric(u_x,u_x)u_x,\quad x\in S^1;\\
&u(x,0)=u_0^i(x).
\end{aligned}\right.
\end{equation}
Here $u_0^i\in C^\infty(S^1, N)$ and $\|u_0^i-u_0\|_{H^4}\rightarrow
0$. By (ii) in Lemma {\ref{lm:2.1}} we know that for each $i$ and
any $k\geq 5$, (\ref{eq:2.35}) admits a local solution $u^i\in
L^\infty([0, T^{\max}_i), H^k(S^1, N))$, where
$T_i^{\max}=T_i^{\max}(S^1, \|u_0^i\|_{H^5})$ is the maximal
existence interval of $u^i$.

As $N$ may not be compact, we let $\Omega_i\triangleq \{p\in
N:\text{dist}_N(p,u_0^i(S^1))<1\}$, which is an open subset of $N$
with compact closure $\bar{\Omega}_i$. Denote
$$\Omega_\infty\triangleq \{p\in
N:\text{dist}_N(p,u_0(S^1))<1\}\quad\mbox{and}\quad\Omega_0
\triangleq \{p\in N: \text{dist}_N(p,\Omega_\infty)<1\}.$$ Since
$\|u_0^i-u_0\|_{H^4}\rightarrow 0$, then $\Omega_i \subset\subset
\Omega_0$ as $i$ is large enough. Let
 $$T'_i=\sup\{t>0: u^i(S^1,t)\subset\Omega_i\}.$$
By the same argument as in Lemma \ref{lm:2.1} we can show that there
holds true for all $t\in[0,T_i]$
\begin{eqnarray*}{d\over
dt} \left(-{1\over\ld}\sum_{s=0}^2\int
Ric(\nabla_x^su^i_x,\nabla_x^su^i_x)dx\right)\leq{C(\Omega_0,\lambda)}
\sum_{l=2}^4\left(-{1\over\ld}\sum_{s=0}^2\int
Ric(\nb_x^su^i_x,\nb_x^su^i_x)dx\right)^l.
\end{eqnarray*}
If we let $f^i(t)=-{1\over\ld}\sum_{s=0}^2\int
Ric(\nabla_x^su^i_x,\nabla_x^su^i_x)dx+1$ and
$g^i(t)=||u^i_x||^2_{H^2}+1$, then we have
\begin{eqnarray}\label{eq:2.36}
{df^i\over dt}\leq C(\Omega_0)(f^i)^4.
\end{eqnarray}
Moreover, since $\Omega_i\subset\subset \Omega_0$ and
$\overline{\Omega_0}$ is compact, we have
$$f^i(0)=-{1\over\ld}\sum_{s=0}^2\int Ric(\nb_x^su_{0x}^i,\nb_x^su_{0x}^i)dx
+1\leq C(\Omega_0,\ld)||u^i_{0x}||_{H^2}^2+1 .$$

It follows from the above differential inequality (\ref{eq:2.36})
that there holds true
$$g^i(t)\leq f^i(t)\leq \left(\frac{(f^i(0))^3}{1-3(f^i(0))^3C(\Omega_0,\ld)t}\right)^{\frac{1}{3}},$$
as
$$t<\frac{1}{3(f^i(0))^3C(\Omega_0,\ld)}.$$
Then, there exists constants
\begin{eqnarray*}0<T_0^i&=&T_0^i\left(\Omega_0,
||u_{0x}||_{H^2},\lambda\right)\\
&=&\frac{1}{4(\lambda||u_{0x}||_{H^2}^2+1)^3C(\Omega_0,\ld)}\leq\frac{1}{4(f^i(0))^3C(\Omega_0,\ld)},
\end{eqnarray*}
and $C^i_0= 4^{\frac{1}{3}}f^i(0)>0$ such that
\begin{eqnarray*}
\lambda||u^i_{x}||^2_{H^2}\leq \sum_{s=0}^2\int
Ric(\nb_x^su^i_{0x},\nb_x^su^i_{0x})dx \leq C^i_0\leq
\widetilde{C}^i_0,\quad t\in[0,\text{min}(T^i_0,T'_i)],
\end{eqnarray*}
where
$\widetilde{C}^i_0=4^{\frac{1}{3}}\left(C(\Omega_0,\ld)||u^i_{0x}||_{H^2}^2+1\right).$

 For $k\geq 3$, there exists $$0<C^i_{k-2}=C_{k-2}\left(\lambda,-{1\over\ld}\sum_{s=0}^k\int
Ric(\nb_x^su^i_{0x},\nb_x^su^i_{0x})dx\right)\leq
%\widetilde{C}^i_{k-2}=
\widetilde{C}_{k-2}\left(\lambda,\Omega_0,||u_{0x}^i||_{H^2}\right)$$
such that for $t\in[0,\text{min}(T^i_0,T'_i)]$
\begin{eqnarray*}
||u^i_{x}||_{H^k}\leq -{1\over\ld}\sum_{s=0}^k\int
Ric(\nb_x^su^i_{0x},\nb_x^su^i_{0x})dx\leq
C^i_{k-2}\leq\widetilde{C}_{k-2}\left(\lambda,\Omega_0,||u_{0x}^i||_{H^2}\right).
\end{eqnarray*}
Since $\|u_0^i-u_0\|_{H^4}\rightarrow 0$, when $i$ is large enough
we have
$$T_0=\frac{1}{4(\|u_{0x}\|_{H^2}^2+1 +\delta_0)C(\Omega_0)}<T^i_0,$$
where $\delta_0$ is a small positive number. It is easy to see that,
as $i$ is large enough,

\begin{equation}\label{eq:2.40}\widetilde{C}^i_0\leq
\widetilde{C}_0(\|u_{0x}\|_{H^2})+\delta_0\quad\mbox{and}\quad
\widetilde{C}^i_1\leq \widetilde{C}_1(\|u_{0x}\|_{H^3})+\delta_0.
\end{equation}

Note that we always have $T_i^{\max}> \text{min}(T_0,T'_i)$ when $i$
is large enough. Otherwise, by Lemma \ref{lem:2.2} we can find a
time-local solution $u_1$ of (\ref{eq:0.1}) and $u_1$ satisfies the
initial value condition
$$u_1(x,T_i^{\max}-\epsilon)=u(x,T_i^{\max}-\epsilon),$$
where $0<\epsilon<T_i^{\max}$ is a small number. Then by the local
existence theorem, $u_1$ exists on the time interval
$(T_i^{\max}-\epsilon,T_i^{\max}-\epsilon+\eta)$ for some constant
$\eta>0$. The uniform bounds on $||u_x||_{H^2}$ and $||\nabla_x^m
u_x ||_{L^2}$ (for all $m>2$) implies that $\eta$ is independent of
$\epsilon$. Thus, by choosing $\epsilon$ sufficiently small, we have
$$T_i^e=T_i^{\max}-\epsilon+\eta>T_i^{\max}.$$
By the uniqueness result, we have that $u_1(x,t)=u(x,t)$ for all
$t\in[T_i^{\max}-\epsilon,T_i^e)$. Thus we get a solution of the
Cauchy problem (\ref{eq:0.1}) on the time interval $[0,T_e)$, which
contradicts the maximality of $T_i^{\max}$.

Now we need to show that $T'_i$ have a uniform lower bound as $i$ is
large enough. For each large enough $i$, if $T'_i\geq T_0$ we obtain
the lower bound. Otherwise, by the same argument as in Lemma
\ref{lm:2.1} we have
$$T'_i\geq \frac{1}{\mathcal{M}_i}$$
where $$\mathcal{M}_i= \sup_{[0, \text{min}(T_0,T'_i)]}
||u_t^i||_{L^\infty}\leq C \sup_{[0, \text{min}(T_0,T'_i)]}
||u^i_t||^a_{H^1}\sup_{[0, \text{min}(T_0,T'_i)]}
||u^i_t||^{1-a}_{L^2}\equiv M_i.$$ It should be pointed out that to
derive the estimates $L^\infty$ estimates on $||u^i_t(s)||_{L^2}$
and $||u^i_t(s)||_{H^1}$ we need only to have $u^i\in L^\infty([0,
\text{min}(T_0,T'_i)], H^4(S^1, N))$, since the equation of the
geometric flow (\ref{eq:0.1}) is a third-order dispersive equation.
It is not difficult to see from (\ref{eq:2.40}) that there exists a
positive constant $M(\Omega_0, ||u_{0x}||_{H^3})$ such that, as $i$
is large enough,
 $$M_i(\Omega_0, ||u^i_{0x}||_{H^3})\leq M(\Omega_0, ||u_{0x}||_{H^3}),$$
since $\|u_0^i-u_0\|_{H^4}\rightarrow 0$.

Let $T^*=\text{min}(T_0, \frac{1}{M})$. As $i$ is large enough, we
always have $u^i \in L^\infty([0, T^*], H^4(S^1, N))$. By letting
$i\rightarrow \infty$ and taking the same arguments as in Lemma
\ref{lem:2.2}, we know there exists $u \in L^\infty([0, T^*],
H^4(S^1, N))$ such that
 $$u^i \rightarrow u\quad[\text{weakly}^*]\quad\text{in}\quad L^\infty([0,T^*],H^4(S^1,N))$$
and $u$ is a local solution to (\ref{eq:0.1}). Theorem \ref{thm:1.1}
guarantees that the local solution is unique and it is continuous
with respect to $t$, i.e., $u\in C([0, T^*], H^4(S^1, N))$. Thus, we
finish the proof of the theorem.
\end{proof}

\section{Conversation Laws}  In this section, we let $(N,h)$ be a locally symmetric space
with metric $h$. Then it is easy to see that $\nb Ric\equiv0$ on $N$
since $\nb R\equiv0$. Then for a smooth solution $u(x,t):S^1\times
(0,T)\rightarrow N$ of the Cauchy problem (\ref{eq:0.1}), we will
derive in this section the conservation laws $E_1(u),E_2(u)$ and the
semi-conservation law $E_3(u)$ introduced in Sec.1.

Precisely, we have the following results:
\begin{lem}\label{lm:1.1}
Assume $N$ is a locally symmetric space. If $u:S^1\times
(0,T)\rightarrow N$ is a smooth solution of the Cauchy problem of
the new geometric flow (\ref{eq:0.1}), then
$${dE_1\over dt}={d\over dt}\int Ric(u_x,u_x)dx=0,$$ in other
words, $E_1(u)=E_1(u_0)$ for all $t\in(0,T)$.
\end{lem}
\begin{proof} With the assumption of $N$, we have $\nabla Ric=0$.
Hence
\begin{eqnarray*}
{dE_1\over dt}&=&{d\over dt}\int Ric(u_x,u_x)dx
=2\int Ric(\nb_tu_x,u_x)\\
&=&2\int Ric(\nabla_x u_t,u_x) dx=-2\int Ric(\nb_xu_x,u_t).
\end{eqnarray*}
Substituting (\ref{eq:0.1}) into above yields
\begin{eqnarray*}
{dE_1\over dt}&=&-2\int Ric(\nabla_x u_x, \nabla_x^2 u_x) dx-2\ro\int Ric(\nabla_x u_x, u_x)Ric(u_x,u_x) dx\\
&=&-\int\nb_x(Ric(\nb_xu_x,\nb_xu_x))-{\ro\over2}\int\nb_x\left(Ric(u_x,u_x)^2\right)\\
&=&0.
\end{eqnarray*}
This completes the proof. \end{proof}

\begin{lem}\label{lm:3.2}
Assume $N$ is a locally symmetric space. If $u:S^1\times
(0,T)\rightarrow N$ is a smooth solution of the Cauchy problem of
the new geometric flow (\ref{eq:0.1}) and let
$$E_2(u)=\int Ric(\nabla_x u_x,\nabla_x u_x)dx-{\ro\over2}\int Ric(u_x,u_x)^2 dx.$$
Then $E_2(u)$ is conserved and we have $${d\over dt}E_2(u)=0.$$
\end{lem}

\begin{proof} We start by differentiating each term of $E_2(u)$ with respect
to $t$. For the first term of $E_2$, after integrating by parts,  we
have
\begin{eqnarray*}
&&{}{d\over dt}\int Ric( \nabla_x u_x,\nabla_x u_x) dx
= 2\int Ric(\nabla_t\nb_xu_x,\nabla_x u_x) dx\\
&=&2\int Ric(\nabla_x\nb_tu_x,\nabla_x u_x) dx+2\int Ric( R(u_t,u_x)u_x,\nabla_x u_x) dx\\
&=&2\int Ric(\nabla^2_xu_t,\nabla_x u_x) dx+2\int Ric( R(u_t,u_x)u_x,\nabla_x u_x) dx\\
&=&2\int Ric( \nabla_x^3 u_x,u_t) dx+2\int Ric( u_t, R(\nabla_x
u_x,u_x)u_x) dx.
\end{eqnarray*}
Substituting the equation (\ref{eq:0.1}) we get
\begin{eqnarray}\label{eq:1.5}
&&{}{d\over dt}\int Ric( \nabla_x u_x,\nabla_x u_x) dx\nn\\
&=&2\int Ric(
\nabla_x^3u_x,\nabla_x^2u_x) dx+2\int Ric(\nabla_x^2u_x,R(\nabla_xu_x,u_x)u_x) dx\nn\\
&&{}+2\ro\int Ric(  Ric(u_x,u_x)u_x,R(\nabla_x u_x,u_x)u_x) dx\nn\\
&&{}+2\ro\int Ric(\nabla_x^3u_x,u_x)Ric(u_x,u_x) dx.
\end{eqnarray}
The first three terms of right hand side of (\ref{eq:1.5}) vanish.
In fact
\begin{eqnarray}
2\int Ric( \nabla_x^3u_x,\nabla_x^2u_x) dx&=&\int\nabla_x\left(Ric(
\nabla_x^2u_x,\nabla_x^2u_x)\right) dx=0\nn
\end{eqnarray}
and
\begin{eqnarray*}
&& \int Ric(\nabla_x^2u_x,R(\nabla_x u_x,u_x)u_x) dx
={1\over2}\int\nb_x\left( Ric(\nabla_x u_x,R(\nabla_x
u_x,u_x)u_x)\right)dx=0
\end{eqnarray*}
 since $\nabla_x Ric=\nb_xR=0$ and $R(\nb_xu_x,\nb_xu_x)=0$.

For the third term of (\ref{eq:1.5}), by the property of Ricci
curvature we obtained before, we have
\begin{eqnarray*}
&&{}\int Ric(  Ric(u_x,u_x)u_x,R(\nabla_x u_x,u_x)u_x) dx\\
&=&\int Ric(u_x,R(\nabla_x u_x,u_x)u_x)\cdot Ric(u_x,u_x) dx=0.
\end{eqnarray*}

Thus, for the last term of (\ref{eq:1.5}), integrating by parts
yields
\begin{eqnarray}\label{eq:1.2}
&&{}{d\over dt}\int Ric( \nabla_x u_x,\nabla_x u_x) dx=2\ro\int Ric(\nabla_x^3u_x,u_x)Ric(u_x,u_x) dx\nn\\
&=&-2\ro\int Ric (\nb_x^2u_x,\nb_xu_x)Ric(u_x,u_x)dx-4\ro\int Ric(\nb_x^2u_x,u_x)Ric(\nb_xu_x,u_x)dx\nn\\
&=&6\ro\int Ric(\nb_x u_x,\nb_xu_x)Ric(\nb_xu_x,u_x)dx.
\end{eqnarray}

Now we consider the second term of $E_2$ and differentiate it with
respect to $t$
\begin{eqnarray*}
&&{}{d\over dt}\int Ric(u_x,u_x)^2 dx\\
&=&4\int Ric(\nb_tu_x,u_x)Ric(u_x,u_x)dx=4\int Ric(\nb_xu_t,u_x)Ric(u_x,u_x)dx\nn\\
&=&-4\int Ric(u_t,\nb_xu_x)Ric(u_x,u_x)dx-8\int
Ric(u_t,u_x)Ric(\nb_xu_x,u_x)dx.
\end{eqnarray*}
Substituting (\ref{eq:0.1}) yields
\begin{eqnarray}\label{eq:1.3}
&&{}-{\ro\over2}{d\over dt}\int Ric(u_x,u_x)^2 dx\nn\\
&=&2\ro\int Ric(\nb_x^2u_x,\nb_xu_x)Ric(u_x,u_x)dx+4\ro\int Ric(\nb_x^2u_x,u_x)Ric(\nb_xu_x,u_x)dx\nn\\
&&{}+{6\ro^2}\int Ric(u_x,\nb_xu_x)|Ric(u_x,u_x)|^2dx.
\end{eqnarray}
Note that after integrating by parts we have
\begin{eqnarray}
&&{}2\ro\int Ric(\nb_x^2u_x,\nb_xu_x)Ric(u_x,u_x)dx=-2\ro\int Ric(\nb_xu_x,\nb_xu_x)Ric(\nb_xu_x,u_x)dx\nn\\
&=& 2\ro\int
Ric(\nb_x^2u_x,u_x)Ric(\nb_xu_x,u_x)dx;\nn\end{eqnarray} while
\begin{eqnarray}
6\ro^2\int Ric(u_x,\nb_xu_x)|Ric(u_x,u_x)|^2dx=\ro^2\int \nb_x\big(
Ric(u_x,u_x)^3\big)dx=0.\nn
\end{eqnarray}
This together with (\ref{eq:1.3}) yields
\begin{eqnarray}\label{eq:1.4}
-{\ro\over2}{d\over dt}\int Ric(u_x,u_x)^2 dx=-6\ro\int
Ric(\nb_xu_x,\nb_xu_x)Ric(\nb_xu_x,u_x)dx.
\end{eqnarray}
Combining (\ref{eq:1.3}) and (\ref{eq:1.4}) we have
$${dE_2\over dt}={d\over dt}\int Ric(\nabla_x u_x,\nabla_x u_x) dx-{\ro\over2}{d\over dt}\int Ric(u_x,u_x)^2 dx=0.$$
This completes the proof.\end{proof}

\begin{rem} It is worthy to point out that the
 two conservation laws hold with the assumption that $N$ is a locally symmetric space.
We do not need the condition that the Ricci curvature on $N$ has a
positive lower bound or a negative upper bound. However, this
condition is required such that the following semi-conservation law
would hold true.
\end{rem}
Precisely, we have

\begin{lem}\label{lm:3.3}
Assume $N$ is a locally symmetric space and the Ricci curvature on
$N$ has a positive lower bound $\lambda>0$ (or a negative upper
bound $-\lambda<0$). If $u:S^1\times (0,T)\rightarrow N$ is a smooth
solution of the Cauchy problem of the new geometric flow
(\ref{eq:0.1}) and let
\begin{eqnarray}\label{eq:E_3}
&&{}E_3(u)=6\int Ric(\nb_x^2u_x,\nb_x^2u_x)dx-20\ro\int Ric(\nb_xu_x,u_x)^2dx\nn\\
&&{}-10\ro\int Ric(\nb_xu_x,\nb_xu_x)Ric(u_x,u_x)dx-4\int
Ric(\nb_xu_x,R(\nb_xu_x,u_x)u_x)dx.
\end{eqnarray}
Then we have
$${dE_3\over dt}\leq C(E_3+1),$$ where $C$ is a constant depends on
$N$, $\lambda$, $E_1(u_0)$ and $||\nabla_x u_x||_{L^2}$.

\end{lem}
\begin{proof}
For simplicity, we denote
\begin{eqnarray}
&&{}E_3(u)=A_1F_1+A_2F_2+A_3F_3+A_4F_4,
\end{eqnarray}

where $A_1=6,A_2=-20\ro,A_3=-10\ro,A_4=-4$ and
\begin{eqnarray}
F_1&=&\int Ric(\nb_x^2u_x,\nb_x^2u_x)dx;\nn\\
F_2&=&\int Ric(\nb_xu_x,u_x)^2dx;\nn\\
F_3&=&\int Ric(\nb_xu_x,\nb_xu_x)Ric(u_x,u_x)dx;\nn\\
F_4&=&\int Ric(\nb_xu_x,R(\nb_xu_x,u_x)u_x)dx.
\end{eqnarray}

To begin with, we calculate ${d\over dt}F_1$:
\begin{eqnarray}
&&{}{d\over dt }F_1={d\over dt}\int Ric(\nb_x^2u_x,\nb_x^2u_x)dx=2\int Ric(\nb_t\nb_x^2u_x,\nb_x^2u_x)dx\nn\\
&=&2\int Ric(\nb_x\nb_t\nb_xu_x,\nb_x^2u_x)dx+2\int Ric(R(u_t,u_x)\nb_xu_x,\nb_x^2u_x)dx\nn\\
&=&2\int Ric(\nb^3_xu_t,\nb_x^2u_x)dx+2\int Ric(\nb_x(R(u_t,u_x)u_x),\nb_x^2u_x)dx\nn\\
&&{}+2\int Ric(R(u_t,u_x)\nb_xu_x,\nb_x^2u_x)dx.\nn
\end{eqnarray}
Integrating by parts and substituting (\ref{eq:0.1}) yields
\begin{eqnarray}\label{eq:1.6}
&&{}{d\over dt}\int Ric(\nb_x^2u_x,\nb_x^2u_x)dx\nn\\
&=&-2\int Ric(\nb^5_xu_x,\nb_x^2u_x)dx-2\ro\int Ric(Ric(u_x,u_x)u_x,\nb_x^5u_x)\nn\\
&&{}-2\int Ric(R(\nb_x^2u_x,u_x)u_x,\nb_x^3u_x)dx\nn\\
&&{}+2\int Ric(R(\nb^2_xu_x,u_x)\nb_xu_x,\nb_x^2u_x)dx.
\end{eqnarray}
Here we used the property $R(Ric(u_x,u_x)u_x,u_x)\equiv0$. After
integrating by parts, we could see that the firs term of
(\ref{eq:1.6}) on the right vanishes, while the third term is equal
to the fourth term, i.e.
$$-2\int Ric(R(\nb_x^2u_x,u_x)u_x,\nb_x^3u_x)dx
=2\int Ric(R(\nb^2_xu_x,u_x)\nb_xu_x,\nb_x^2u_x)dx.$$ For the second
term of (\ref{eq:1.6}) , by the calculation of (\ref{eq:2.22}) we
have
\begin{eqnarray}\label{eq:1.7}
&&{}-2\ro\int Ric(\nb_x^5u_x,Ric(u_x,u_x)u_x)\nn\\
%&=&2\int Ric(\nb_x^4u_x,u_x)Ric(\nb_xu_x,u_x)dx+\int Ric(\nb_x^4u_x,\nb_xu_x)Ric(u_x,u_x)dx\nn\\
%&=&-2\int Ric(\nb_x^3u_x,u_x)Ric(\nb^2_xu_x,u_x)dx-2\int Ric(\nb_x^3u_x,u_x)Ric(\nb_xu_x,\nb_xu_x)dx\nn\\
%&&{}-4\int Ric(\nb_x^3u_x,\nb_xu_x)Ric(\nb_xu_x,u_x)dx-\int Ric(\nb_x^3u_x,\nb_x^2u_x)Ric(u_x,u_x)dx\nn\\
&=&20\ro\int Ric(\nb_x^2u_x,\nb_xu_x)Ric(\nb^2_xu_x,u_x)dx\nn\\
&&{}+10\ro\int Ric(\nb_x^2u_x,\nb^2_xu_x)Ric(\nb_xu_x,u_x)dx.
%&&{}+6\int Ric(\nb_x^2u_x,\nb_xu_x)Ric(\nb_xu_x,\nb_xu_x)dx.\nn
\end{eqnarray}
%The last term of (\ref{eq:1.7}) vanishes because
%\begin{eqnarray}\label{eq:100}
%\int Ric(\nb_x^2u_x,\nb_xu_x)Ric(\nb_xu_x,\nb_xu_x)dx={1\over4}\int
%\nb_x\left(|Ric(\nb_xu_x,\nb_xu_x)|^2\right)dx=0.
%\end{eqnarray}

Hence we have
\begin{eqnarray}\label{eq:1.8}
&&{}{d\over dt }F_1={d\over dt}\int Ric(\nb_x^2u_x,\nb_x^2u_x)dx\nn\\
&=&20\ro\int Ric(\nb_x^2u_x,\nb_xu_x)Ric(\nb^2_xu_x,u_x)dx+10\ro\int Ric(\nb_x^2u_x,\nb^2_xu_x)Ric(\nb_xu_x,u_x)dx\nn\\
&&{}+4\int Ric(R(\nb^2_xu_x,u_x)\nb_xu_x,\nb_x^2u_x)dx.
\end{eqnarray}

For the second term of $E_3$,  we calculate
\begin{eqnarray}\label{eq:1.9}
&&{}{d\over dt}F_2={d\over dt}\int Ric(\nb_xu_x,u_x)^2dx\nn\\
&=&2\int Ric (\nb_t\nb_xu_x,u_x)Ric(\nb_xu_x,u_x)dx+2\int Ric (\nb_xu_x,\nb_tu_x)Ric(\nb_xu_x,u_x)dx\nn\\
&=&2\int Ric (\nb_x^2u_t,u_x)Ric(\nb_xu_x,u_x)dx+2\int Ric (\nb_xu_x,\nb_xu_t)Ric(\nb_xu_x,u_x)dx\nn\\
&&{}+2\int Ric(R(u_t,u_x)u_x,u_x)Ric(\nb_xu_x,u_x)dx.
\end{eqnarray}
Note that the last term of (\ref{eq:1.9}) vanishes since
$$Ric(R(u_t,u_x)u_x,u_x)=Ric(R(u_x,u_x)u_x,u_t)=0.$$
Substituting (\ref{eq:0.1}) into (\ref{eq:1.9}) we have:
\begin{eqnarray}\label{eq:1.10}
&&{}{d\over dt}F_2={d\over dt}\int Ric(\nb_xu_x,u_x)^2dx\nn\\
&=&2\int Ric (\nb_x^4u_x,u_x)Ric(\nb_xu_x,u_x)dx+2\int Ric (\nb_x^3u_x,\nb_xu_x)Ric(\nb_xu_x,u_x)dx\nn\\
&&{}+2\ro\int Ric(\nb_x^2(Ric(u_x,u_x)u_x),u_x)Ric(\nb_xu_x,u_x)dx\nn\\
&&{}+2\ro\int
Ric(\nb_x(Ric(u_x,u_x)u_x),\nb_xu_x)Ric(\nb_xu_x,u_x)dx.
\end{eqnarray}
After integrating by parts, we obtain
\begin{eqnarray}\label{eq:1.11}
&&{}2\int Ric (\nb_x^4u_x,u_x)Ric(\nb_xu_x,u_x)dx\nn\\
&=&-2\int Ric (\nb_x^3u_x,\nb_xu_x)Ric(\nb_xu_x,u_x)dx-2\int Ric (\nb_x^3u_x,u_x)Ric(\nb_x^2u_x,u_x)dx\nn\\
&&{}-2\int Ric (\nb_x^3u_x,u_x)Ric(\nb_xu_x,\nb_xu_x)dx\nn\\
&=&2\int Ric (\nb_x^2u_x,\nb^2_xu_x)Ric(\nb_xu_x,u_x)dx\nn\\&&{}+8\int Ric (\nb_x^2u_x,\nb_xu_x)Ric(\nb^2_xu_x,u_x)dx\nn\\
&&{}+4\int Ric (\nb_x^2u_x,\nb_xu_x)Ric(\nb_xu_x,\nb_xu_x)dx.\nn\\
&=&2\int Ric
(\nb_x^2u_x,\nb^2_xu_x)Ric(\nb_xu_x,u_x)dx\nn\\&&{}+8\int Ric
(\nb_x^2u_x,\nb_xu_x)Ric(\nb^2_xu_x,u_x)dx.
\end{eqnarray}
Similarly, for the other three terms in (\ref{eq:1.10}), we could
get the follow results:
\begin{eqnarray}\label{eq:1.12}
&&{}2\int Ric (\nb_x^3u_x,\nb_xu_x)Ric(\nb_xu_x,u_x)dx\nn\\
&=&-2\int Ric (\nb_x^2u_x,\nb^2_xu_x)Ric(\nb_xu_x,u_x)dx-2\int Ric (\nb_x^2u_x,\nb_xu_x)Ric(\nb_x^2u_x,u_x)dx\nn\\
&&{}-2\int Ric (\nb_x^2u_x,\nb_xu_x)Ric(\nb_xu_x,\nb_xu_x)dx;
\end{eqnarray}
\begin{eqnarray}\label{eq:1.13}
&&{}2\ro\int Ric(\nb_x^2(Ric(u_x,u_x)u_x),u_x)Ric(\nb_xu_x,u_x)dx\nn\\
%&=&3\int Ric(\nb_x^2u_x,u_x)Ric(\nb_xu_x,u_x)Ric(u_x,u_x)dx\nn\\
&=&2\ro\int |Ric(\nb_xu_x,u_x)|^3dx-2\ro\int
Ric(\nb_xu_x,\nb_xu_x)Ric(\nb_xu_x,u_x)Ric(u_x,u_x)dx;
\end{eqnarray}
and
\begin{eqnarray}\label{eq:1.14}
&&{}2\ro\int Ric(\nb_x(Ric(u_x,u_x)u_x),\nb_xu_x)Ric(\nb_xu_x,u_x)dx\nn\\
&=&4\ro\int |Ric(\nb_xu_x,u_x)|^3dx+2\ro\int
Ric(\nb_xu_x,\nb_xu_x)Ric(\nb_xu_x,u_x)Ric(u_x,u_x)dx.
\end{eqnarray}
In view of (\ref{eq:1.10})$-$(\ref{eq:1.14}) we have
\begin{eqnarray}\label{eq:1.15}
&&{}{d\over dt}F_2={d\over dt}\int Ric(\nb_xu_x,u_x)^2dx\nn\\
&=&6\int Ric (\nb_x^2u_x,\nb_xu_x)Ric(\nb^2_xu_x,u_x)dx+6\ro\int
|Ric(\nb_xu_x,u_x)|^3dx.
\end{eqnarray}
 Next we compute ${d\over dt}F_3$.
\begin{eqnarray}\label{eq:1.16}
&&{}{d\over dt}F_3={d\over dt}\int Ric(\nb_xu_x,\nb_xu_x)Ric(u_x,u_x)dx\nn\\
&=&{}2\int Ric(\nb_t\nb_xu_x,\nb_xu_x)Ric(u_x,u_x)dx+2\int Ric(\nb_xu_x,\nb_xu_x)Ric(\nb_tu_x,u_x)dx\nn\\
&=&{}2\int Ric(\nb^2_xu_t,\nb_xu_x)Ric(u_x,u_x)dx+2\int Ric(\nb_xu_x,\nb_xu_x)Ric(\nb_xu_t,u_x)dx\nn\\
&&{}+2\int Ric(R(u_t,u_x)u_x,\nb_xu_x)Ric(u_x,u_x)dx.
\end{eqnarray}
Substituting (\ref{eq:0.1}) into (\ref{eq:1.16}) and integrating by
parts yield
\begin{eqnarray}\label{eq:1.17}
&&{}{d\over dt}F_3={d\over dt}\int Ric(\nb_xu_x,\nb_xu_x)Ric(u_x,u_x)dx\nn\\
&=&{}6\int Ric (\nb_x^2u_x,\nb_x^2u_x)Ric(\nb_xu_x,u_x)dx-4\ro\int |Ric(\nb_xu_x,u_x)|^3dx\nn\\
&&{}+10\ro\int Ric(\nb_xu_x,\nb_xu_x)Ric(\nb_xu_x,u_x)Ric(u_x,u_x)dx\nn\\
&&{}-4\ro\int Ric (\nb_xu_x,R(\nb_xu_x,u_x)u_x)Ric(\nb_xu_x,u_x)dx.
\end{eqnarray}

For the fourth term of $E_3(u)$, computing ${d\over dt}F_4$ yields
\begin{eqnarray}\label{eq:1.18}
&&{}{d\over dt}F_4={d\over dt}\int Ric(\nb_xu_x,R(\nb_xu_x,u_x)u_x)dx\nn\\
&=&2\int Ric(\nb_t\nb_xu_x,R(\nb_xu_x,u_x)u_x)dx+2\int Ric(\nb_xu_x,R(\nb_xu_x,u_x)\nb_tu_x)dx\nn\\
&=&2\int Ric(\nb_x^2u_t,R(\nb_xu_x,u_x)u_x)dx+2\int Ric(R(u_t,u_x)u_x,R(\nb_xu_x,u_x)u_x)dx\nn\\
&&{}+2\int Ric(\nb_xu_x,R(\nb_xu_x,u_x)\nb_xu_t)dx\nn\\
&=&2\int Ric(\nb_x^2u_t,R(\nb_xu_x,u_x)u_x)dx+2\int Ric(R(u_t,u_x)u_x,R(\nb_xu_x,u_x)u_x)dx\nn\\
&&{}-2\int Ric(\nb_x^2u_x,R(\nb_xu_x,u_x)u_t)dx-2\int
Ric(\nb_xu_x,R(\nb_x^2u_x,u_x)u_t)dx.
\end{eqnarray}
We substitute (\ref{eq:0.1}) into each term of (\ref{eq:1.18}) on
the right and obtain that:

for the first term of (\ref{eq:1.18})
\begin{eqnarray}\label{eq:1.19}
&&{}2\int Ric(\nb_x^2u_t,R(\nb_xu_x,u_x)u_x)dx\nn\\
&=&2\int Ric(\nb_x^4u_x,R(\nb_xu_x,u_x)u_x)dx+2\ro\int Ric(\nb_x^2(Ric(u_x,u_x)u_x),R(\nb_xu_x,u_x)u_x)dx\nn\\
&=&-2\int Ric(\nb_x^3u_x,R(\nb^2_xu_x,u_x)u_x)dx-2\int Ric(\nb_x^3u_x,R(\nb_xu_x,u_x)\nb_xu_x)dx\nn\\
&&{}+2\ro\int Ric(\nb_x^2u_x,R(\nb_xu_x,u_x)u_x)Ric(u_x,u_x)dx\nn\\
&&{}+8\ro\int Ric(\nb_xu_x,R(\nb_xu_x,u_x)u_x)Ric(\nb_xu_x,u_x)dx\nn\\
&=&4\int Ric(\nb_x^2u_x,R(\nb^2_xu_x,u_x)\nb_xu_x)dx%+2\int Ric(\nb_x^2u_x,R(\nb^2_xu_x,u_x)\nb_xu_x)dx\nn\\
+6\ro\int Ric(\nb_xu_x,R(\nb_xu_x,u_x)u_x)Ric(\nb_xu_x,u_x)dx.\nn\\
\end{eqnarray}
For the last three terms of (\ref{eq:1.18}) we have
\begin{eqnarray}\label{eq:1.20}
&&{}2\int Ric(R(u_t,u_x)u_x,R(\nb_xu_x,u_x)u_x)dx-2\int Ric(\nb_x^2u_x,R(\nb_xu_x,u_x)u_t)dx\nn\\
&&{}-2\int Ric(\nb_xu_x,R(\nb_x^2u_x,u_x)u_t)dx\nn\\
&=&2\int Ric(R(\nb_x^2u_x,u_x)u_x,R(\nb_xu_x,u_x)u_x)dx-2\int Ric(\nb_xu_x,R(\nb_x^2u_x,u_x)\nb_x^2u_x)dx\nn\\
&&{}-2\ro\int Ric(\nb_x^2u_x,R(\nb_xu_x,u_x)u_x)Ric(u_x,u_x)dx\nn\\
&&{}-2\ro\int Ric(\nb_xu_x,R(\nb_x^2u_x,u_x)u_x)Ric(u_x,u_x)dx\nn\\
&=&2\int Ric(\nb_x^2u_x,R(\nb_x^2u_x,u_x)\nb_xu_x)dx-2\int Ric(R(\nb_xu_x,u_x)\nb_xu_x,R(\nb_xu_x,u_x)u_x)dx\nn\\
&&{}+4\ro\int Ric(\nb_xu_x,R(\nb_xu_x,u_x)u_x)Ric(\nb_xu_x,u_x)dx.
\end{eqnarray}
Combining (\ref{eq:1.18}), (\ref{eq:1.19}) and (\ref{eq:1.20})
yields
\begin{eqnarray}\label{eq:1.21}
&&{}{d\over dt}F_4={d\over dt}\int Ric(\nb_xu_x,R(\nb_xu_x,u_x)u_x)dx\nn\\
&=&6\int Ric(\nb_x^2u_x,R(\nb_x^2u_x,u_x)\nb_xu_x)dx-2\int Ric(R(\nb_xu_x,u_x)\nb_xu_x,R(\nb_xu_x,u_x)u_x)dx\nn\\
&&{}+10\ro\int Ric(\nb_xu_x,R(\nb_xu_x,u_x)u_x)Ric(\nb_xu_x,u_x)dx.
\end{eqnarray}
Hence by (\ref{eq:1.8}), (\ref{eq:1.15}), (\ref{eq:1.17}) and
(\ref{eq:1.21})
 we obtain
 \begin{eqnarray}\label{eq:E3}
&&{}{d\over dt}E_3(u)={d\over dt}(\sum_{i=1}^4A_i F_i)\nn\\
&=&(20\ro A_1+6A_2)\int Ric(\nb_x^2u_x,\nb_xu_x)Ric(\nb^2_xu_x,u_x)dx\nn\\
&&{}+(10\ro A_1+6A_3)\int Ric(\nb_x^2u_x,\nb^2_xu_x)Ric(\nb_xu_x,u_x)dx\nn\\
&&{}+(4A_1+6A_4)\int Ric(\nb_x^2u_x,R(\nb_x^2u_x,u_x)\nb_xu_x)dx\nn\\
%&&{}+6A_2\int Ric (\nb_x^2u_x,\nb_xu_x)Ric(\nb^2_xu_x,u_x)dx\nn\\
%&&{}+6A_3\int Ric (\nb_x^2u_x,\nb_x^2u_x)Ric(\nb_xu_x,u_x)dx\nn\\
%&&{}-2A_3\int |Ric(\nb_xu_x,u_x)|^3dx\nn\\
&&{}+2\ro(5A_4-2A_3)\int Ric (\nb_xu_x,R(\nb_xu_x,u_x)u_x)Ric(\nb_xu_x,u_x)dx\nn\\
&&{}+2\ro(3A_2-2A_3)\int |Ric(\nb_xu_x,u_x)|^3dx\nn\\
&&{}+10\ro A_3\int Ric(\nb_xu_x,\nb_xu_x)Ric(\nb_xu_x,u_x)Ric(u_x,u_x)dx\nn\\
%&&{}+6A_4\int Ric(\nb_x^2u_x,R(\nb_x^2u_x,u_x)\nb_xu_x)dx\nn\\
&&{}-2A_4\int Ric(R(\nb_xu_x,u_x)\nb_xu_x,R(\nb_xu_x,u_x)u_x)dx.%\nn\\
%&&{}+5A_4\int Ric(\nb_xu_x,R(\nb_xu_x,u_x)u_x)Ric(\nb_xu_x,u_x)dx.
 \end{eqnarray}
Since $A_1=6$, $A_2=-20\ro$, $A_3=-10\ro$, $A_4=-4$ , it easy to see
that the first three terms of (\ref{eq:E3}) vanish. Hence
\begin{eqnarray}%\label{eq:E3}
{d\over dt}E_3(u)&=&-80\ro^2\int |Ric(\nb_xu_x,u_x)|^3dx\nn\\
&&{}40\ro(-1+\ro)\int Ric (\nb_xu_x,R(\nb_xu_x,u_x)u_x)Ric(\nb_xu_x,u_x)dx\nn\\
&&{}-100\ro^2\int Ric(\nb_xu_x,\nb_xu_x)Ric(\nb_xu_x,u_x)Ric(u_x,u_x)dx\nn\\
&&{}+8\int Ric(R(\nb_xu_x,u_x)\nb_xu_x,R(\nb_xu_x,u_x)u_x)dx\nn\\
&\leq& C(N,\ro)\int |\nb_xu_x|^3|u_x|^3dx\nn\\
&\leq& C(N,\ro)||u_x||^3_{L^\infty}\int|\nabla_x u_x|^3dx.\nn
 \end{eqnarray}
Using the following interpolation inequalities (see \cite{YD} for
details)
\begin{eqnarray}
||u_x||_{L^\infty}&\leq& C(N)(||\nabla_x u_x||^2_{L^2}+||u_x||^2_{L^2})^{1\over4}||u_x||^{1\over2}_{L^2}\nn\\
&\leq& C(N, ||\nabla_x u_x||_{L^2}, E_1(u_0));\label{eq:3.33}\\
||\nabla_x u_x||^3_{L^3}&\leq& C(N)(||\nabla_x ^2u_x||^2_{L^2}+||\nabla_x u_x||^2_{L^2})^{1\over4}||\nabla_x u_x||^{5\over2}_{L^2}\nn\\
&\leq&C(N, ||\nabla_x
u_x||_{L^2})\big(1+||\nabla_x^2u_x||^2_{L^2}\big),\nn\label{eq:3.34}
\end{eqnarray}
we have
\begin{eqnarray} \label{eq:3.e3}
{dE_3\over dt}&\leq& C\big(1+\int|\nabla_x^2u_x|^2dx),
%&\leq& C(1+E_3),
\end{eqnarray}
where $C=C(N, ||\nabla_x u_x||_{L^2}, E_1(u_0),\ro)>0$ only depends
on $N$, $E_1(u_0)$ and $||\nabla_x u_x||_{L^2}$. By the assumption,
if $Ric\geq\ld>0$, we have
$$\int|\nabla_x^2u_x|^2dx\leq {1\over \lambda}\int Ric(\nb_x^2u_x,\nb_x^2u_x)dx\leq C_1E_3+C_2,$$
where $C_1,C_2$ only depend on $N,\lambda$ and $||\nabla_x
u_x||_{L^2}$. This together with (\ref{eq:3.e3}) yields
\begin{eqnarray} \label{eq:e3}
{dE_3\over dt}&\leq& C\big(1+E_3),
%&\leq& C(1+E_3),
\end{eqnarray}
where $C=C(N, \lambda, ||\nabla_x u_x||_{L^2}, E_1(u_0))>0$ only
depends on $N$,$\lambda$, $E_1(u_0)$ and $||\nabla_x u_x||_{L^2}$.

It is easy to check that for the case the Ricci curvature on $N$ has
a negative upper bound $-\lambda<0$, we could also get the result
via a same argument. Hence we complete the proof of the lemma.
\end{proof}

\section{Global existence}
In this section we derive the global existence of the Cauchy problem
(\ref{eq:0.1}) and prove Theorem \ref{thm:1.3}. Since $u_0\in
H^4(S^1, N)$, we can always choose a sequence of smooth maps $u_{0i}
\in C^\infty(S^1, N)$ such that, as $i\rightarrow\infty$,
$$\|u_{0i}-u_0\|_{H^4}\rightarrow 0.$$
By the arguments in Theorem \ref{thm:1.2}, we get that the Cauchy
problem (\ref{eq:0.1}) with the initial map $u_{0i}$ admits a unique
smooth local solution $u^i$ such that
 $$u^i \in C([0, T(N, \|u_{0i}\|_{H^4})], H^k(S^1, N)))$$
for any $k \geq 4$. Obviously, it is easy to see that $T(N,
\|u_{0i}\|_{H^4})$ have a uniform lower bound. Hence, letting
$i\rightarrow \infty$, we obtain the local solution to the Cauchy
problem of the new geometric flow with the initial map $u_0\in
H^4(S^1, N)$. Thus, to prove Theorem \ref{thm:1.3}, we only need to
consider the case $u_0$ is a smooth map from $S^1$ into $N$.

Let $u$ be the local smooth solution of (\ref{eq:0.1}) which exists
on the maximal time interval $[0,T)$. We only discuss the case that
$T<\infty$ and we assume the Ricci curvature on $N$ has a positive
lower bound $\lambda>0$.

From Lemma \ref{lm:1.1}, we know that the energy is bounded by
$E_1(u_0)$ because
$$0\leq\int|u_x|^2dx\leq{1\over\lambda}E_1(u(t))={1\over\lambda}E_1(u_0),\qquad \text{for any}\quad t\in [0,T).$$
Moreover, from Lemma \ref{lm:3.2} we know that $E_2$ is preserved,
that is
$$E_2(u)=\int Ric(\nabla_x u_x,\nabla_x u_x)dx-{\ro\over2}\int Ric( u_x,u_x)^2dx=E_2(u_0)$$
 Thus we have
\begin{eqnarray}\label{eq:4.1}
||\nabla_x u_x||^2_{L^2}&\leq& {1\over\lambda}\int Ric(\nabla_x u_x,\nabla_x u_x)dx Ric\nn\\
 &=&{1\over\lambda}\left(E_3(u_0)+{\ro\over2}\int Ric( u_x,u_x)^2dx\right)\nn\\
&\leq& C(N,\lambda,\ro)\left(E_3(u_0)+\int|u_x|^4 dx\right)\nn\\
&\leq& C(N,E_1(u_0),\lambda,E_3(u_0),\ro),
\end{eqnarray}
note that here we used the interpolation inequality
\begin{eqnarray}\label{eq:4.2}
||u_x||^4_{L^4}&\leq&(||\nabla_x u_x||^2_{L^2}+||u_x||^2_{L^2})^{1\over2}||u_x||^3_{L^2}\nn\\
&\leq&{1\over 2}||\nabla_x u_x||^2_{L^2}+C(E_1(u_0),\lambda).\nn
\end{eqnarray}

Then, (\ref{eq:4.1}) together with Lemma \ref{lm:3.3} yields
$${dE_3\over dt}\leq C(N,E_1(u_0),E_2(u_0),\ro)\big(1+E_3\big).$$
%Integrating the above inequality with respect to $t$ yields
%\begin{eqnarray*}
%E_4(u(t))\leq E_4(u_0)+C\int_0^tE_4dt+Ct.
%\end{eqnarray*}
By Gronwall inequality, we get that $E_3(u(t))$ is uniformly bounded
on $[0,T)$. Hence, we obtain
\begin{eqnarray}\label{eq:4.4}
6\lambda \int ||\nb_x^2u_x||^2dx&\leq&6\int Ric(\nabla_x^2u_x,\nabla_x^2u_x)dx\nn\\&=&E_3(u)+10\int Ric(\nb_xu_x,u_x)^2dx\nn\\
&&{}+5\int Ric(\nb_xu_x,\nb_xu_x)Ric(u_x,u_x)dx\nn\\
&&{}+4\int
Ric(\nb_xu_x,R(\nb_xu_x,u_x)u_x)dx\nn\\
&\leq& C(N,\lambda,
E_1(u_0),E_2(u_0),E_3(u_0),\ro)+C(N)||u_x||^2_{L^\infty}||\nabla_x
u_x||^2_{L^2}.
\end{eqnarray}

In view of (\ref{eq:3.33}), (\ref{eq:4.1}) and the boundedness of
$E_3$, we see that $||\nabla_x^2u_x||_{L^2}$ is uniformly bounded on
$[0,T)$. Hence we have
$$\sup_{t\in[0,T)} ||u_x||_{H^2}\leq C(N,\lambda,E_1(u_0),E_2(u_0),E_3(u_0),\ro).$$

It follows from the proof of Lemma \ref{lm:2.1} that for $m>2$
$$\sup_{t\in[0,T)} ||\nabla_x ^m u_x||_{L^2}\leq C(N,||u_{0x}||_{L^2}, ||\nabla_x u_{0x}||_{L^2},
||\nabla_x^2 u_{0x}||_{L^2},\cdots,||\nabla_x^m
u_{0x}||_{L^2},\ro).$$

Thus, if $T$ is finite, we can find a time-local solution $u_1$ of
(\ref{eq:0.1}) and $u_1$ satisfies the initial value condition
$$u_1(x,T-\epsilon)=u(x,T-\epsilon),$$
where $0<\epsilon<T$ is a small number. Then by the local existence
theorem, $u_1$ exists on the time interval
$(T-\epsilon,T-\epsilon+\eta)$ for some constant $\eta>0$. The
uniform bounds on $||u_x||_{H^2}$ and $||\nabla_x^m u_x ||_{L^2}$
(for all $m>2$) implies that $\eta$ is independent of $\epsilon$.
Thus, by choosing $\epsilon$ sufficiently small, we have
$$T_1=T-\epsilon+\eta>T.$$
By the uniqueness result, we have that $u_1(x,t)=u(x,t)$ for all
$t\in[T-\epsilon,T_1)$. Thus we get a solution of the Cauchy problem
(\ref{eq:0.1}) on the time interval $[0,T_1)$,
which contradicts the maximality of $T$.

For the case that the Ricci curvature on $N$ has a negative upper
bound $-\lambda<0$, we still could obtain
the results via the same arguments and we omit the details. Thus we complete the
proof of Theorem {\ref{thm:1.3}}. \hspace*{\fill}$\Box$\\

\vspace{1cm}
\noindent{Xiaowei Sun}\\
School of Statistics and Mathematics, \\
Central University of Finance and Economics\\
Beijing 100081, P.R. China.\\
Email: sunxw@cufe.edu.cn\\\\
Youde Wang\\
Academy of Mathematics and Systems Science\\
Chinese Academy of Sciences,\\
Beijing 100190, P.R. China.\\
Email: wyd@math.ac.cn

\end{document}